\documentclass{article} 
\usepackage{nips15submit_e,times}
\usepackage{hyperref}
\usepackage{url}
\usepackage{amssymb,amsmath,amsthm}
\usepackage{mathtools}
\usepackage{pifont}
\usepackage{algorithmic,algorithm}
\usepackage{multirow}

\newcommand{\vsb}{\vspace*{-0.1cm}}

\title{A Universal Catalyst for First-Order Optimization}

\author{Hongzhou Lin$^1$,~~ Julien Mairal$^1$ ~~and~~ Zaid Harchaoui$^{1,2}$\\
$^1$Inria ~~~~~~~$^2$NYU \\
\texttt{\{hongzhou.lin,julien.mairal\}@inria.fr} \\ \texttt{zaid.harchaoui@nyu.edu} 
}

\nipsfinalcopy 

 \DeclareMathOperator*{\argmin}{arg\,min}
 
 \newcommand{\reg}{\psi}
 \newcommand{\mtd}{\mathcal{M}}
 \newcommand{\C}{\mathcal{C}}
 \newcommand{\R}{\mathbb{R}}
 
 \newcommand{\E}{\mathbb{E}}
 \newcommand{\p}{\mathbb{P}}
 \newcommand{\surr}{d}
 \newcommand{\Surr}{D}
 \newcommand{\barz}{\bar{z}}
 \renewcommand\leqslant\leq
 \renewcommand\geqslant\geq
 \renewcommand\epsilon\varepsilon
 \renewcommand\ln\log
 \renewcommand\star{*}
\def\defin{\triangleq}
\def\Real{{\mathbb R}}
 \theoremstyle{plain}\newtheorem{thm}{Theorem}[section]
 \theoremstyle{definition}\newtheorem{defi}[thm]{Definition}
 \theoremstyle{plain}\newtheorem{prop}[thm]{Proposition}
 \theoremstyle{plain}\newtheorem{lemma}[thm]{Lemma}
 \theoremstyle{definition}
 \theoremstyle{plain}
 \theoremstyle{plain}
 
 \newcommand{\prox}{\mathrm{prox}}
 \newcommand{\Acal}{\mathcal{A}}
 \newcommand{\boots}{catalyst~}

 \newcommand{\Otilde}{\tilde{O}}
 
\usepackage{pifont}
 \newcommand{\vs}{\vspace*{-0.1cm}}

\begin{document}
\maketitle

\vspace*{-0.25cm}
\begin{abstract}
   We introduce a generic scheme for accelerating first-order optimization methods
in the sense of Nesterov, which builds upon a new analysis of the
accelerated proximal point algorithm. Our approach consists of minimizing a
convex objective by approximately solving a sequence of well-chosen auxiliary
problems, leading to faster convergence. This strategy applies to a large class
of algorithms, including gradient descent, block coordinate descent, SAG, SAGA,
SDCA, SVRG, Finito/MISO, and their proximal variants.  For all of these
methods, we provide acceleration and explicit support for non-strongly
convex objectives. In addition to theoretical speed-up, we also show that
acceleration is useful in practice, especially for ill-conditioned problems
where we measure significant improvements.

\end{abstract}

\section{Introduction}\label{sec:intro}
A large number of machine learning and signal processing problems are
formulated as the minimization of a composite objective
function~$F: \Real^p \to \Real$:
\begin{equation}
   \min_{x \in \Real^p} \left\{ F(x) \defin f(x) + \reg(x) \right\}, \label{eq:obj}
\end{equation}
where $f$ is convex and has Lipschitz continuous derivatives with constant~$L$
and~$\reg$ is convex but may not be differentiable.
The variable $x$ represents model parameters 
and the role of~$f$ is to ensure that the estimated parameters fit some observed data.
Specifically, $f$ is often a large sum of functions
\begin{equation}
   f(x) \defin \frac{1}{n} \sum_{i=1}^{n} f_i(x), \label{eq:obj2}
\end{equation}
and each term $f_i(x)$ measures the fit between~$x$ and a data point indexed
by~$i$. The function~$\reg$ in~(\ref{eq:obj}) acts as a regularizer; it is
typically chosen to be the squared~$\ell_2$-norm, which is smooth, or to be a
non-differentiable penalty such as the~$\ell_1$-norm or another
sparsity-inducing norm~\cite{bach2012optimization}. Composite minimization 
also encompasses constrained minimization if we consider
extended-valued indicator functions~$\psi$ that may take the value $+\infty$
outside of a convex set~$\C$ and~$0$ inside (see \cite{hiriart1996convex}).

Our goal is to accelerate gradient-based or \emph{first-order} methods that are designed
to solve~(\ref{eq:obj}), with a particular focus on large sums of
functions~(\ref{eq:obj2}). By ``accelerating'', we mean generalizing a
mechanism invented by Nesterov~\cite{nesterov} that improves the
convergence rate of the gradient descent algorithm. More precisely, when
$\reg=0$, gradient descent steps produce iterates~$(x_k)_{k \geq 0}$ such that
$F(x_k) - F^\star = O(1/k)$, where $F^\star$ denotes the minimum value of~$F$.
Furthermore, when the objective~$F$ is strongly convex with
constant $\mu$, the rate of convergence becomes linear in $O( (1- {\mu}/{L})^k
)$.  These rates were shown by Nesterov~\cite{nesterov1983} to be suboptimal
for the class of first-order methods, and instead optimal rates---$O(1/k^2)$
for the convex case and $O( (1- \sqrt{{\mu}/{L}})^k )$ for the
$\mu$-strongly convex one---could be obtained by taking gradient steps at
well-chosen points.  Later, this acceleration technique was extended to deal
with non-differentiable regularization
functions~$\psi$~\cite{fista,nesterov2013gradient}.

For modern machine learning problems involving a large sum of~$n$ functions, a
recent effort has been devoted to developing fast \emph{incremental}
algorithms~\cite{saga,finito,miso,sag,sdca,proxsvrg} that can exploit the
particular structure of~(\ref{eq:obj2}). Unlike full gradient approaches which
require computing and averaging 
$n$ gradients $\nabla f(x) = (1/n)
\sum_{i=1}^n\nabla f_i(x)$ at every iteration, incremental techniques have a cost per-iteration that is independent of~$n$. 
The price to pay is the need to store a moderate amount of information regarding past iterates, but the
benefit is significant in terms of computational complexity. 

\vs
\paragraph{Main contributions.} 
Our main achievement is a \emph{generic acceleration scheme} that applies to a large
class of optimization methods.  By analogy with substances that increase
chemical reaction rates, we call our approach a ``catalyst''.
A method may be accelerated if it has
linear convergence rate for strongly convex problems. This is
the case for
full gradient \cite{fista,nesterov2013gradient} and
block coordinate descent methods~\cite{nesterov2012,richtarik2014}, which already
have well-known accelerated variants. More importantly, it also applies to 
incremental algorithms such as SAG~\cite{sag},
SAGA~\cite{saga}, Finito/MISO~\cite{finito,miso}, SDCA~\cite{sdca}, and
SVRG~\cite{proxsvrg}. Whether or not these methods could be accelerated
was an important open question. It was only known to be the case for dual
coordinate ascent approaches such as SDCA~\cite{accsdca} or
SDPC~\cite{zhangxiao} for strongly convex objectives. Our work provides a
universal positive answer regardless of the strong convexity
of the objective, which brings us to our second achievement.

Some approaches such as Finito/MISO, SDCA, or SVRG are only defined
for strongly convex objectives. A classical trick to apply them to general
convex functions is to add a small regularization
$\varepsilon \|x\|^2$~\cite{sdca}. The drawback of this strategy is that it requires choosing in advance the parameter~$\varepsilon$, which is related to
the target accuracy.  A consequence of our work is to
automatically provide a \emph{direct support for non-strongly convex objectives}, thus
removing the need of selecting~$\varepsilon$ beforehand.

\vs
\paragraph{Other contribution: Proximal MISO.} The approach Finito/MISO, which was
proposed in~\cite{finito} and~\cite{miso}, is an incremental
technique for solving smooth unconstrained $\mu$-strongly convex problems when $n$ is larger than a constant $\beta
L/\mu$ (with $\beta=2$ in~\cite{miso}). In addition to providing acceleration
and support for non-strongly convex objectives, we also make the following specific contributions: \\
\hspace*{0.2cm}$\bullet$~we extend the method
and its convergence proof to deal with the composite problem~(\ref{eq:obj});\\
\hspace*{0.2cm}$\bullet$~we fix the method to remove the ``big data condition'' $n \geq \beta
L/\mu$.\\
The resulting algorithm can be interpreted as a variant of proximal SDCA~\cite{sdca}
with a different step size and a more practical optimality certificate---that is, checking the optimality condition does not require
evaluating a dual objective. Our construction is indeed purely \emph{primal}. 
Neither our proof of convergence nor the algorithm use duality, while SDCA is originally a dual ascent technique.

\vs
\paragraph{Related work.}
The catalyst acceleration can be interpreted as a variant of the proximal point
algorithm~\cite{Bauschke:2011,guler:1992}, which is a central concept in convex optimization,
underlying augmented Lagrangian approaches, and composite minimization schemes~\cite{bertsekas:2015,Parikh13}.
The proximal point algorithm consists of solving~(\ref{eq:obj}) by minimizing a
sequence of auxiliary problems involving a quadratic regularization term. In general,
these auxiliary problems cannot be solved
with perfect accuracy, and several notations of inexactness were proposed,
including~\cite{guler:1992,he2012accelerated,salzo2012inexact}.
The \boots approach hinges upon (i) an acceleration technique for the proximal
point algorithm originally introduced in the pioneer work~\cite{guler:1992}; (ii)
a more practical inexactness criterion than those proposed in the
past.\footnote{Note that our inexact criterion was also studied, among others,
in~\cite{salzo2012inexact}, but the analysis of~\cite{salzo2012inexact} led to
the conjecture that this criterion was too weak to warrant acceleration. Our
analysis refutes this conjecture.}
As a result, we are able to control the rate of convergence for approximately solving the
auxiliary problems with an optimization method~$\mtd$. In turn, we are also able to obtain
the computational complexity of the global procedure for
solving~(\ref{eq:obj}), which was not possible with previous
analysis~\cite{guler:1992,he2012accelerated,salzo2012inexact}.
When instantiated in different first-order optimization settings, our 
analysis yields systematic acceleration.

Beyond~\cite{guler:1992}, several works have inspired this paper. In particular,
accelerated SDCA~\cite{accsdca} is an instance of an inexact accelerated
proximal point algorithm, even though this was not explicitly stated
in~\cite{accsdca}. Their proof of convergence relies on different tools than
ours. Specifically, we use the concept of \emph{estimate sequence} from
Nesterov~\cite{nesterov}, whereas the direct proof of~\cite{accsdca}, in the
context of SDCA, does not extend to non-strongly convex objectives.
Nevertheless, part of their analysis proves to be helpful to obtain our main
results.
Another useful methodological contribution was the convergence analysis of inexact proximal
gradient methods of~\cite{proxinexact}.
Finally, similar ideas appear in the independent work~\cite{frostig}. Their results overlap 
in part with ours, but both papers adopt different directions. Our analysis is for
instance more general and provides  support for non-strongly convex
objectives. Another independent work with related results is~\cite{conjugategradient}, which introduce 
an accelerated method for the minimization of finite sums, which is not based on the proximal point algorithm.

\section{The Catalyst Acceleration}\label{sec:algorithm}
We present here our generic acceleration
scheme, which can operate on any first-order or gradient-based
optimization algorithm with linear convergence rate for strongly convex
objectives. 

\vs
\paragraph{Linear convergence and acceleration.}
Consider the problem~(\ref{eq:obj}) with a $\mu$-strongly convex function~$F$, 
where the strong convexity is defined with respect to the $\ell_2$-norm.
A minimization algorithm~$\mtd$, generating the sequence of iterates $(x_k)_{k
\geq 0}$, has a \textit{linear convergence rate} if there exists $\tau_{\mtd,
F}$ in $(0,1)$ and a constant $C_{\mtd,F}$ in $\R$ such that
\begin{equation}\label{mtd}
   F(x_k) - F^* \leq  C_{\mtd,F}(1-\tau_{\mtd, F})^k, 
\end{equation}
where $F^*$ denotes the minimum value of $F$. 
The quantity~$\tau_{\mtd,F}$ controls the convergence rate: the larger is $\tau_{\mtd, F}$, the
faster is convergence to $F^*$. However, for a given algorithm $\mtd$, the quantity
$\tau_{\mtd, F}$ depends usually on the ratio~$L/\mu$, which is often called
the \emph{condition number} of~$F$.

The \boots acceleration is a general approach that allows to wrap algorithm
$\mtd$ into an accelerated algorithm $\Acal$, which enjoys a faster linear
convergence rate, with $\tau_{\Acal,F} \geq \tau_{\mtd, F}$. 
As we will also see, the \boots acceleration may also be useful when~$F$ is not
strongly convex---that is, when~$\mu=0$. In that case, we may even consider a
method~$\mtd$ that requires strong convexity to operate, and obtain an
accelerated algorithm~$\Acal$ that can minimize~$F$ with near-optimal
convergence rate~$\Otilde(1/k^2)$.\footnote{In this paper, we use the notation
$O(.)$  to hide constants. The notation~$\Otilde(.)$ also hides logarithmic
factors.}

Our approach can accelerate a wide range of first-order optimization
algorithms, starting from classical gradient descent. It also applies to
randomized algorithms such as SAG, SAGA, SDCA, SVRG and Finito/MISO,
whose rates of convergence are
given in expectation.
Such methods should be contrasted with stochastic gradient methods~\cite{nemirovski,juditsky_nemirovsky_2012}, which minimize a different non-deterministic
function. Acceleration of stochastic gradient methods is beyond the
scope of this work. 

\vs
\paragraph{Catalyst action.}
We now highlight the mechanics of the catalyst algorithm, which is presented in
Algorithm~\ref{alg:catalyst}.  It consists of replacing, at iteration $k$, the original objective function $F$ by an auxiliary
objective $G_k$, close to $F$ up to a quadratic term: 
\begin{equation}
   G_k(x) \defin F(x) + \frac{\kappa}{2} \Vert x- y_{k-1} \Vert^2, \label{eq:gk}
\end{equation}
where $\kappa$ will be specified later and~$y_k$ is obtained by an
extrapolation step described in~(\ref{y_k}).  Then, at iteration $k$, the
accelerated algorithm $\Acal$ minimizes $G_k$ up to accuracy $\varepsilon_k$. 

Substituting~(\ref{eq:gk}) to~(\ref{eq:obj}) has two consequences. 
On the one hand, minimizing~(\ref{eq:gk}) only provides an approximation of the solution of~(\ref{eq:obj}),
unless~$\kappa=0$; on the other hand, the auxiliary objective~$G_k$ enjoys a
better condition number than the original objective~$F$, which makes it easier
to minimize. 
For instance, when $\mtd$ is the regular gradient descent algorithm with~$\reg=0$,
$\mtd$ has the rate of convergence~(\ref{mtd}) for minimizing $F$ 
with~$\tau_{\mtd,F} = \mu/L$.
However, owing to the additional quadratic term, 
$G_k$ can be minimized by~$\mtd$ with the rate~(\ref{mtd}) where
$\tau_{\mtd,G_k} = (\mu+\kappa)/(L+\kappa) > \tau_{\mtd,F}$.
In practice, there exists an ``optimal'' choice for~$\kappa$, which controls
the time required by~$\mtd$ for solving the auxiliary
problems~(\ref{eq:gk}), and the quality of approximation of~$F$ by the functions~$G_k$.
This choice will be driven by the convergence analysis in Sec.~\ref{subsec:strongly}-\ref{convergence2};
see also Sec.~\ref{sec:accel-factor-calculs} for special cases.

\vs
\paragraph{Acceleration via extrapolation and inexact minimization.}
Similar to the classical gradient descent scheme of
Nesterov~\cite{nesterov}, Algorithm~\ref{alg:catalyst} involves an
extrapolation step~(\ref{y_k}). As a consequence,
the solution of the auxiliary problem~(\ref{eq:approx}) at iteration~$k+1$ is
driven towards the extrapolated variable~$y_k$.
As shown in~\cite{guler:1992}, this step is in fact sufficient to reduce the
number of iterations of Algorithm~\ref{alg:catalyst} to solve~(\ref{eq:obj})
when~$\varepsilon_k=0$---that is, for running the \emph{exact} accelerated
proximal point algorithm.

Nevertheless, to control the total computational complexity of an accelerated
algorithm~$\Acal$, it is necessary to take into account the complexity of
solving the auxiliary problems~(\ref{eq:approx}) using~$\mtd$. This is where
our approach differs from the classical proximal point algorithm
of~\cite{guler:1992}. Essentially, both algorithms are the same, but we use the
weaker inexactness criterion~$G_k(x_k)-G_k^\star \leq \varepsilon_k$, where
the sequence $(\varepsilon_k)_{k \geq 0}$ is fixed beforehand, and only depends on the initial point.
This subtle difference has important consequences: (i) in practice, this
condition can often be checked by computing duality gaps; (ii) in theory,
the methods~$\mtd$ we consider have linear convergence rates, which allows
us to control the complexity of step~(\ref{eq:approx}), and
then to provide the computational complexity of~$\Acal$.

\begin{algorithm}[hbtp]
   \caption{Catalyst}\label{alg:catalyst}
    \begin{algorithmic}[1]
    \INPUT initial estimate $x_0 \in \Real^p$, parameters $\kappa$ and~$\alpha_0$, sequence  $(\varepsilon_k)_{k \geq 0}$, optimization method~$\mtd$;
    \STATE Initialize $q = \mu/(\mu + \kappa)$ and $y_0 = x_0$;
    \WHILE {the desired stopping criterion is not satisfied}
    \STATE Find an approximate solution of the following problem using~$\mtd$
    \begin{equation}
       x_{k}  \approx \argmin_{x \in \Real^p}  \left \{ G_{k}(x) \defin F(x)+ \frac{\kappa}{2} \| x - y_{k-1}\|^2  \right \}~~~\text{such that}~~~G_k(x_k) - G_k^\star \leq \varepsilon_k. \label{eq:approx}
    \end{equation}
    \STATE Compute $\alpha_k \in (0,1)$ from equation $\alpha_{k}^2  = (1-\alpha_{k}) \alpha_{k-1}^2 + q \alpha_{k}$;
    \STATE Compute 
    \begin{equation} \label{y_k} 
    y_{k} = x_{k}+ \beta_k (x_{k} - x_{k-1}) \quad \text{ with } \quad \beta_k = \frac{\alpha_{k-1}(1-\alpha_{k-1})}{\alpha_{k-1}^2+ \alpha_{k}}.
    \end{equation}
    \ENDWHILE
    \OUTPUT $x_k$ (final estimate).
 \end{algorithmic}
\end{algorithm}
\vs

\section{Convergence Analysis}\label{sec:convergence}
In this section, we present the theoretical properties of Algorithm~\ref{alg:catalyst},
for optimization methods~$\mtd$ with deterministic convergence rates of the
form~(\ref{mtd}). When the rate is given as an expectation, a simple
extension of our analysis described in Section~\ref{sec:applications} is needed.
For space limitation reasons, we shall sketch the proof mechanics here, and
defer the full proofs to Appendix~\ref{sec:proofs}.

\vs
\subsection{Analysis for $\mu$-Strongly Convex Objective Functions}\label{subsec:strongly}
We first analyze the convergence rate of Algorithm~\ref{alg:catalyst} for solving problem~\ref{eq:obj}, 
regardless of the complexity required to solve the subproblems~(\ref{eq:approx}).
We start with the $\mu$-strongly convex case. 
\begin{thm}[\bfseries Convergence of Algorithm~\ref{alg:catalyst}, $\mu$-Strongly Convex Case]~\label{convergence}\newline
 Choose  $\alpha_0 = \sqrt{q}$ with~$q=\mu/(\mu+\kappa)$ and
$$  \epsilon_k = \frac{2}{9} (F(x_0) -F^*) (1- \rho)^{k} \quad  \text{ with } \quad  \rho < \sqrt{q}.$$
Then, Algorithm~\ref{alg:catalyst} generates iterates $(x_k)_{k \geq 0}$ such that
\begin{equation} \label{rate of F}
F(x_{k}) -F^* \leqslant C (1-\rho)^{k+1} (F(x_0) -F^*) \quad  \text{ with } \quad  C= \frac{8}{(\sqrt{q} -\rho)^2}.
\end{equation}
\end{thm}
This theorem characterizes the linear convergence rate of
Algorithm~\ref{alg:catalyst}. It is worth noting that the choice of~$\rho$ is
left to the discretion of the user, but it can safely be set to
$\rho=0.9\sqrt{q}$ in practice. The choice~$\alpha_0=\sqrt{q}$ was made for
convenience purposes since it leads to a simplified analysis, but larger values
are also acceptable, both from theoretical and practical point of views. 
Following an advice from~Nesterov\cite[page 81]{nesterov} originally dedicated
to his classical gradient descent algorithm, we may for instance recommend
choosing~$\alpha_0$ such that~$\alpha_0^2 + (1-q)\alpha_0 - 1 = 0$. 

The choice of the sequence $(\varepsilon_k)_{k \geq 0}$ is also subject to
discussion since the quantity $F(x_0) -F^*$ is unknown beforehand.
Nevertheless, an upper bound may be used instead, which will only affects the
corresponding constant in~(\ref{rate of F}).  Such upper bounds 
can typically be obtained by computing a duality gap at $x_0$, or by using
additional knowledge about the objective. For instance, when $F$ is
non-negative, we may simply choose $\varepsilon_k=(2/9)F(x_0)(1-\rho)^k$.

The proof of convergence uses the concept of estimate sequence invented by
Nesterov~\cite{nesterov}, and introduces an extension to deal with the
errors~$(\varepsilon_k)_{k \geq 0}$. To control the accumulation of errors, we borrow the
methodology of~\cite{proxinexact} for inexact
proximal gradient algorithms. Our construction yields a convergence
result that encompasses both strongly convex and non-strongly convex cases. Note
that estimate sequences were also used in~\cite{guler:1992}, but, as noted by~\cite{salzo2012inexact}, the proof of~\cite{guler:1992} only applies
when using an extrapolation step~(\ref{y_k}) that involves the true minimizer
of~(\ref{eq:approx}), which is unknown in practice.  To obtain a rigorous
convergence result like~(\ref{rate of F}), a different approach was needed.

Theorem~\ref{convergence} is important, but it does not provide yet the global
computational complexity of the full algorithm, which includes 
the number of iterations performed by~$\mtd$ for approximately 
solving the auxiliary problems~(\ref{eq:approx}).  The next proposition
characterizes the complexity of this inner-loop.
\begin{prop}[\bfseries Inner-Loop Complexity, $\mu$-Strongly Convex Case]~\label{prop:complexity}\newline
   Under the assumptions of Theorem~\ref{convergence}, let us consider a method $\mtd$
   generating iterates~$(z_t)_{t \geq 0}$ for minimizing the function~$G_k$
   with linear convergence rate of the form
    \begin{equation}
       G_k(z_{t}) - G_k^\star \leq A (1-\tau_{\mtd})^t (G_k(z_0) - G_k^\star). \label{eq:rategk}
    \end{equation}
    When $z_0=x_{k-1}$, the precision~$\varepsilon_k$ is reached with a number
    of iterations~$T_{\mtd}=\Otilde(1/\tau_{\mtd})$, where the notation $\Otilde$ hides some
    universal constants and some logarithmic dependencies in~$\mu$ and~$\kappa$.
\end{prop}
This proposition is generic since the assumption~(\ref{eq:rategk}) is relatively
standard for gradient-based methods~\cite{nesterov}.  It may now be used to obtain
the global rate of convergence of an accelerated algorithm. 
By calling~$F_s$ the objective function value obtained after performing $s=k T_\mtd$ iterations of the
method~$\mtd$, the true convergence rate of the accelerated algorithm~$\Acal$ is 
\begin{equation}
   F_s - F^\star = F\left(x_{\frac{s}{T_\mtd}}\right) -F^*   \leqslant C (1-\rho)^{\frac{s}{T_\mtd}} (F(x_0) -F^*)
		      \leqslant C \left (1- \frac{\rho}{T_\mtd} \right )^{s} (F(x_0) -F^*).
\end{equation}
As a result, algorithm~$\Acal$ has a global linear rate of convergence  with parameter
$$\tau_{\Acal, F} = {\rho}/{T_\mtd} = \Otilde(\tau_{\mtd}\sqrt{\mu}/{\sqrt{\mu+\kappa}}),$$
where~$\tau_{\mtd}$ typically depends on~$\kappa$ (the greater, the faster is~$\mtd$).
Consequently, $\kappa$ will be chosen to maximize the ratio $\tau_{\mtd}/{\sqrt{\mu+\kappa}}$. 
Note that for other algorithms~$\mtd$ that do not satisfy~(\ref{eq:rategk}),
additional analysis and possibly a different initialization~$z_0$ may be
necessary (see Appendix~\ref{appendix:miso} for example).

\subsection{Convergence Analysis for Convex but Non-Strongly Convex Objective Functions}\label{subsec:notstrongly}
We now state the convergence rate when the objective is \emph{not strongly convex}, that is when $\mu = 0$.

\begin{thm}[\bfseries Convergence of Algorithm~\ref{alg:catalyst}, Convex, but Non-Strongly Convex Case]~\label{convergence2}\newline
   When~$\mu=0$, choose $\alpha_0= (\sqrt{5}-1)/2$ and 
   \begin{equation}\label{eq:epsk2}
   \epsilon_k =  \frac{2(F(x_0) -F^*)}{9(k+2)^{4+\eta }}  \quad \text{ with } \eta >0. 
\end{equation}
Then, Algorithm~\ref{alg:catalyst} generates iterates $(x_k)_{k \geq 0}$ such that
\begin{equation} \label{rate of F2}
F(x_{k}) -F^* \leqslant \frac{8}{(k+2)^2} \left ( \left( 1+ \frac{2}{\eta}\right)^2 (F(x_0) -F^*)+\frac{\kappa}{2} \Vert x_0 - x^*\Vert^2 \right ) .
\end{equation}
\end{thm}
This theorem is the counter-part of Theorem~\ref{convergence} when~$\mu=0$.  
The choice of~$\eta$ is left to the discretion of the user; it empirically seem to have very low influence
on the global convergence speed, as long as it is chosen small enough (e.g., we use $\eta=0.1$ in practice).
It shows that Algorithm~\ref{alg:catalyst} achieves the optimal rate of convergence of first-order methods,
but \emph{it does not take into account the complexity of solving the subproblems~(\ref{eq:approx})}. Therefore, we need the following proposition:
\begin{prop}[\bfseries Inner-Loop Complexity, Non-Strongly Convex Case]~\label{prop:complexity2}\newline
   Assume that $F$ has bounded level sets. Under the assumptions of Theorem~\ref{convergence2}, 
   let us consider a method $\mtd$ generating iterates~$(z_t)_{t \geq 0}$ for minimizing the function~$G_k$
   with linear convergence rate of the form~(\ref{eq:rategk}). Then, there
   exists $T_{\mtd}=\Otilde(1/\tau_{\mtd})$, such that for any $k \geqslant 1$, solving $G_k$ with
   initial point $x_{k-1}$ requires at most $T_{\mtd} \log (k+2)$ iterations of
   $\mtd$.
\end{prop}

We can now draw up the global complexity of an accelerated algorithm~$\Acal$
when~$\mtd$ has a linear convergence rate~(\ref{eq:rategk})
for~$\kappa$-strongly convex objectives. To produce $x_k$, $\mtd$ is called at
most $k T_{\mtd} \log(k+2)$ times. 
Using the global iteration counter $s = kT_\mtd\log(k+2)$, we get
\begin{align}
   F_s -F^* & \leqslant \frac{8T_{\mtd}^2\log^2(s)}{s^2} \left ( \left( 1+ \frac{2}{\eta}\right)^2 (F(x_0) -F^*)+\frac{\kappa}{2} \Vert x_0 - x^*\Vert^2 \right ) \; .
\end{align}
If~$\mtd$ is a first-order method, this rate is \emph{near-optimal}, up to a logarithmic factor, when compared to the optimal rate $O(1/s^2)$,
which may be the price to pay for using a generic acceleration scheme.

\section{Acceleration in Practice}\label{sec:applications}
We show here how to accelerate existing algorithms~$\mtd$ and
compare the convergence rates obtained before and after \boots acceleration.
For all the algorithms we consider, 
we study rates of convergence   
in terms of \emph{total number of iterations} (in expectation, when necessary)
to reach accuracy~$\epsilon$. 
We first show how to accelerate 
full gradient and randomized coordinate descent algorithms~\cite{richtarik2014}.
Then,
we discuss other approaches such as SAG~\cite{sag}, SAGA~\cite{saga}, or
SVRG~\cite{proxsvrg}.
Finally, we present a new proximal version of the incremental gradient approaches 
Finito/MISO~\cite{finito,miso}, along with its accelerated version. 
Table~\ref{tab:accel-summary} summarizes 
the acceleration obtained for the algorithms considered. 

\vs
\paragraph{Deriving the global rate of convergence.}
The convergence rate of an accelerated 
algorithm $\mathcal{A}$ is driven by the parameter $\kappa$. 
In the strongly convex case, the best choice is the one that maximizes 
the ratio $\tau_{\mtd, G_k}/\sqrt{\mu + \kappa}$.
As discussed in Appendix~\ref{sec:accel-factor-calculs}, this rule also holds
when~(\ref{eq:rategk}) is given in expectation and in many cases where the
constant~$\C_{\mtd,G_k}$ is different than $A(G_k(z_0) -
G_k^\star)$ from~(\ref{eq:rategk}).
When~$\mu=0$, the choice of~$\kappa > 0$ only affects
the complexity by a multiplicative constant. A rule of thumb is to maximize the ratio
$\tau_{\mtd, G_k}/\sqrt{L + \kappa}$ (see Appendix~\ref{sec:accel-factor-calculs} for more details).

After choosing~$\kappa$, the global iteration-complexity is given by
$\text{Comp} \leq k_{\text{in}} \: k_{\text{out}}$, where $k_{\text{in}}$ is an
upper-bound on the number of iterations performed by~$\mtd$ per inner-loop, and
$k_{\text{out}}$ is the upper-bound on the number of outer-loop iterations,
following from Theorems~\ref{convergence}-\ref{convergence2}. 
Note that for simplicity, we always consider that
$L \gg \mu$ such that we may write $L-\mu$ simply as~``$L$'' in the convergence
rates.

\subsection{Acceleration of Existing Algorithms}

\paragraph{Composite minimization.}
Most of the algorithms we consider here, namely the proximal gradient
method~\cite{fista,nesterov2013gradient}, SAGA~\cite{saga},
(Prox)-SVRG~\cite{proxsvrg}, can handle composite objectives with a
regularization penalty $\reg$ that admits a proximal operator~$\prox_{\psi}$, defined for any $z$ as
\begin{equation*}
\prox_{\psi}(z) \defin \argmin_{y \in \mathbb{R}^p} \left\{  \psi(y) + \frac{1}{2} \Vert y- z  \Vert^2  \right\} \; .
\end{equation*}
Table~\ref{tab:accel-summary} presents convergence rates that are valid for proximal and non-proximal 
settings, since most methods we consider are able to deal with such non-differentiable penalties.
The exception is SAG~\cite{sag}, for which proximal variants are not analyzed. 
The incremental method Finito/MISO has also been limited to non-proximal settings so far. 
In~Section~\ref{sec:prox-miso-intro}, we actually introduce the extension of MISO to composite minimization, 
and establish its theoretical convergence rates. 

\vs
\paragraph{Full gradient method.}
A first illustration is the algorithm obtained when accelerating the regular ``full'' gradient descent (FG),
and how it contrasts with Nesterov's accelerated variant (AFG). Here, the optimal
choice for $\kappa$ is $ L -2\mu$. In the strongly convex case, 
we get an accelerated rate of convergence in  $\Otilde(n\sqrt{L/\mu}\log(1/\varepsilon))$, 
which is the same as AFG up to logarithmic terms.
A similar result can also be obtained for randomized coordinate descent methods~\cite{richtarik2014}.

\vs
\paragraph{Randomized incremental gradient.}
We now consider randomized incremental gradient methods, resp. SAG~\cite{sag}
and SAGA~\cite{saga}. When~$\mu > 0$, we focus on the ``ill-conditioned'' setting $n \leq
L/\mu$, where these methods have the complexity
$O((L/\mu)\log(1/\varepsilon))$. Otherwise, their complexity becomes $O(n
\log(1/\varepsilon))$, which is independent of the condition number and seems
theoretically optimal~\cite{agarwal}.

For these methods, the best choice for $\kappa$ has the form
$\kappa  = {a(L-\mu)}/{(n+b)} -\mu$, 
with $(a,b) = (2, -2)$ for SAG, $(a,b) = (1/2, 1/2)$ for SAGA. A similar formula, with a constant~$L'$ in place of $L$, holds for SVRG; we omit it here for brevity. 
SDCA~\cite{accsdca} and Finito/MISO~\cite{finito,miso} are actually
related to incremental gradient methods, and the choice for $\kappa$ has a
similar form with $(a,b) = (1, 1)$.  

\begin{table}
\label{tab:accel-summary}
   \centering
   \footnotesize
   \renewcommand{\arraystretch}{1.4}
   \begin{tabular}{|l|c|c||c|c|}
      \hline
      & Comp. $\mu > 0$ & Comp. $\mu = 0$ & Catalyst $\mu > 0$ & Catalyst $\mu = 0$ \\
      \hline
      FG  & $O\left(n\left(\frac{L}{\mu}\right)\log\left(\frac{1}{\varepsilon}\right)\right)$ & \multirow{3}{*}{$O\left(n\frac{L}{\varepsilon}\right)$} 
      &  $\Otilde\left(n\sqrt{\frac{L}{\mu}}\log\left(\frac{1}{\varepsilon}\right)\right)$ & \multirow{6}{*}{$\Otilde\left(n\frac{L}{\sqrt{\varepsilon}}\right)$} \\
      \cline{1-2}\cline{4-4}
      SAG~\cite{sag}  & \multirow{4}{*}{$O\left(\frac{L}{\mu} \log\left(\frac{1}{\varepsilon}\right)\right)$} &  & \multirow{4}{*}{$\Otilde\left(\sqrt{\frac{n L}{\mu}} \log\left(\frac{1}{\varepsilon}\right)\right)$}&  \\
      \cline{1-1}
      SAGA~\cite{saga} &  &  &  &  \\
      \cline{1-1}\cline{3-3}
      Finito/MISO-Prox &  & \multirow{3}{*}{not avail.} &  &  \\
      \cline{1-1}
      SDCA~\cite{sdca}  &  &  & &  \\
      \cline{1-2}\cline{4-4}
      SVRG~\cite{proxsvrg}  & $O\left(\frac{L'}{\mu} \log\left(\frac{1}{\varepsilon}\right)\right)$ &  &  $\Otilde\left(\sqrt{\frac{nL'}{\mu}} \log\left(\frac{1}{\varepsilon}\right)\right)$ &  \\
      \hline
      \hline
      Acc-FG~\cite{nesterov2013gradient} & $O\left(n\sqrt{\frac{L}{\mu}}\log\left(\frac{1}{\varepsilon}\right)\right)$ 
      & $O\left(n\frac{L}{{\sqrt{\varepsilon}}}\right)$ & \multicolumn{2}{c|}{\multirow{2}{*}{no acceleration}} \\
      \cline{1-3}
      Acc-SDCA~\cite{accsdca} & $\Otilde\left(\sqrt{\frac{nL}{\mu}}\log\left(\frac{1}{\varepsilon}\right)\right)$ &  not avail.
      &  \multicolumn{2}{c|}{} \\
      \hline
   \end{tabular}
   \vs
   \vs
   \caption{Comparison of rates of convergence, before and after the \boots
      acceleration, resp. in the strongly-convex and non strongly-convex cases.
      {\bf{To simplify, we only present the case where $n \leq L/\mu$ when~$\mu >
   0$}}. For all incremental algorithms, there is indeed no acceleration
otherwise. The quantity~$L'$ for SVRG
is the average Lipschitz constant of the functions~$f_i$~(see \cite{proxsvrg}).
}
\vs
\end{table}

\subsection{Proximal MISO and its Acceleration}
\label{sec:prox-miso-intro}
Finito/MISO was proposed in~\cite{finito} and~\cite{miso} for solving the
problem~(\ref{eq:obj}) when~$\psi=0$ and when~$f$ is a sum of $n$ $\mu$-strongly convex
functions~$f_i$ as in~(\ref{eq:obj2}), which are also differentiable
with~$L$-Lipschitz derivatives.
The algorithm maintains a list of quadratic lower bounds---say~$(d_i^k)_{i=1}^n$ at
iteration~$k$---of the functions~$f_i$ and randomly updates one of them at each iteration by using
strong-convexity inequalities. The current iterate~$x_k$ is then obtained by minimizing
the lower-bound of the objective
\begin{equation}
   x_k = \argmin_{x \in \Real^p} \left\{ D_k(x) = \frac{1}{n}\sum_{i=1}^n d_i^k(x) \right\}.
\end{equation}

Interestingly, since~$D_k$ is a lower-bound of~$F$ we also have~$D_k(x_k) \leq
F^\star$, and thus the quantity $F(x_k) - D_k(x_k)$ can be used as an
optimality certificate that upper-bounds $F(x_k)-F^\star$. 
Furthermore, this certificate was shown to converge to zero with a rate similar to SAG/SDCA/SVRG/SAGA under
the condition~$n \geq 2L/\mu$. In this section, we show how to remove this
 condition and how to provide support to non-differentiable
functions~$\psi$ whose proximal operator can be easily computed. We shall
briefly sketch the main ideas, and we refer to
Appendix~\ref{appendix:miso} for a thorough presentation.

The first idea to deal with a nonsmooth regularizer~$\psi$ is to change the definition of~$D_k$:
\begin{equation*}
   D_k(x) = \frac{1}{n}\sum_{i=1}^n d_i^k(x)  + \psi(x), 
\end{equation*}
which was also proposed in~\cite{finito} without a convergence proof. 
Then, because the $d_i^k$'s are quadratic functions, the minimizer~$x_k$ of~$D_k$
can be obtained by computing the proximal operator of~$\psi$ at a particular point.
The second idea to remove the condition~$n \geq 2L/\mu$ is to modify the update
of the lower bounds~$d_i^k$. Assume that index~$i_k$ is selected
among~$\{1,\ldots,n\}$ at iteration~$k$, then
 \begin{equation*}
   d_i^k(x) = \!\left\{ 
      \begin{array}{ll}
         (1-\delta)d_i^{k-1}(x) \!+ \delta(f_i(x_{k-1}) \!+\! \langle \nabla f_i(x_{k-1}) ,x - x_{k-1} \rangle \!+\! \frac{\mu}{2}\|x-x_{k-1}\|^2) & \text{if}~~~  i=i_k \\
         d_i^{k-1}(x) & \text{otherwise} 
      \end{array}\right.
\end{equation*}
Whereas the original Finito/MISO uses~$\delta=1$, our new variant uses
$\delta=\min(1,{\mu n}/{2(L-\mu)})$.  The
resulting algorithm turns out to be very close to variant ``5'' of proximal
SDCA~\cite{sdca}, which corresponds to using a different value
for~$\delta$. The main difference between SDCA and MISO-Prox is that the latter
does not use duality.
It also provides a different (simpler) optimality certificate~$F(x_k)-D_k(x_k)$, which
is guaranteed to converge linearly, as stated in the next theorem.
\begin{thm}[\bfseries Convergence of MISO-Prox]~\label{thm1.1}\newline
   Let $(x_k)_{k \geq 0}$ be obtained by MISO-Prox, then 
   \begin{equation}\label{eq:thm_miso}
      \mathbb{E}[F(x_{k})]  - F^* \leqslant \frac{1}{ \tau} (1-\tau)^{k+1} \left ( F(x_0)  - D_0(x_0) \right )~~\text{with}~~ \tau \geqslant \min \Big\{ \frac{\mu}{4L}, \frac{1}{2n} \Big\}.
\end{equation}
   Furthermore, we also have fast convergence of the certificate
   \begin{displaymath}
      \mathbb{E}[F(x_{k}) - \Surr_k(x_k)]   \leqslant \frac{1}{ \tau} (1-\tau)^{k} \left(F^* - D_0(x_0) \right ).
   \end{displaymath}
\end{thm}
The proof of convergence is given in Appendix~\ref{appendix:miso}.
Finally, we conclude this section by noting that MISO-Prox enjoys the catalyst
acceleration, leading to the iteration-complexity presented in
Table~\ref{tab:accel-summary}. Since the convergence rate~(\ref{eq:thm_miso}) does not
have exactly the same form as~(\ref{eq:rategk}), Propositions~\ref{prop:complexity} and~\ref{prop:complexity2}
cannot be used and additional analysis, given in Appendix~\ref{appendix:miso}, is needed.
Practical forms of the algorithm are also presented there, along
with discussions on how to initialize it.

\section{Experiments}\label{sec:exp}
We evaluate the Catalyst acceleration on three methods that have never
been accelerated in the past: SAG~\cite{sag},
SAGA~\cite{saga}, and MISO-Prox. 
We focus on $\ell_2$-regularized logistic regression, where the
regularization parameter~$\mu$ yields a lower bound on the
strong convexity parameter of the problem. 

We use three datasets used in~\cite{miso},
namely~\textrm{real-sim, rcv1}, and~\textrm{ocr}, which are relatively large, with up
to~$n=2\,500\,000$ points for~\textrm{ocr} and~$p=47\,152$ variables
for~\textrm{rcv1}. We consider three regimes: $\mu=0$ (no
regularization), $\mu/L=0.001/n$ and~$\mu/L=0.1/n$,
which leads significantly larger condition numbers 
than those used in other studies ($\mu/L \approx 1/n$ in~\cite{miso,sag}).
We compare~MISO, SAG, and SAGA with their default parameters, which are recommended by their theoretical analysis (step-sizes
$1/L$ for SAG and~$1/3L$ for SAGA), and study several accelerated variants. 
The values of $\kappa$ and $\rho$ and the sequences~$(\varepsilon_k)_{k \geq
0}$ are those suggested in the previous sections, with~$\eta\!=\!0.1$ in~(\ref{eq:epsk2}).
Other implementation details are presented in Appendix~\ref{appendix:exp}.

{The restarting strategy for $\mtd$ is key to achieve acceleration in practice}. 
All of the methods we compare store~$n$ gradients evaluated at previous iterates of the algorithm. 
We always use the gradients from the previous run of~$\mtd$ to initialize a new one.  {We detail in 
Appendix~\ref{appendix:exp} the initialization for each method.} 
Finally, we evaluated a heuristic that constrain~$\mtd$ to
always perform at most~$n$ iterations (one pass over the data); we call this
variant AMISO2 for MISO whereas AMISO1 refers to the regular ``vanilla''
accelerated variant, and we also use this heuristic to accelerate SAG.

The results are reported in Table~\ref{table:exp}.
We always obtain a huge speed-up for MISO, which suffers from
numerical stability issues when the condition number is {very large} (for
instance, $\mu/L = 10^{-3}/n = 4.10^{-10}$ for~\textrm{ocr}). Here, not only does
the catalyst algorithm accelerate MISO, but it also stabilizes it.
Whereas MISO is slower than SAG and SAGA in this ``small $\mu$'' regime, AMISO2 is almost
systematically the
best performer. We are also able to accelerate SAG and SAGA in general, even
though the improvement is less significant than for MISO. In particular, SAGA without 
acceleration proves to be the best method on~\textrm{ocr}. One reason may be
its ability to adapt to the unknown strong convexity parameter~$\mu' \geq
\mu$ of the objective near the solution. When~$\mu'/L \geq 1/n$, we indeed obtain
a regime where acceleration does not occur (see
Sec.~\ref{sec:applications}). Therefore, this experiment suggests that
adaptivity to unknown strong convexity is of high interest for incremental
optimization.

\vs\vs

\begin{figure}[hbtp!]
   \centering
   ~~\includegraphics[width=0.28\linewidth]{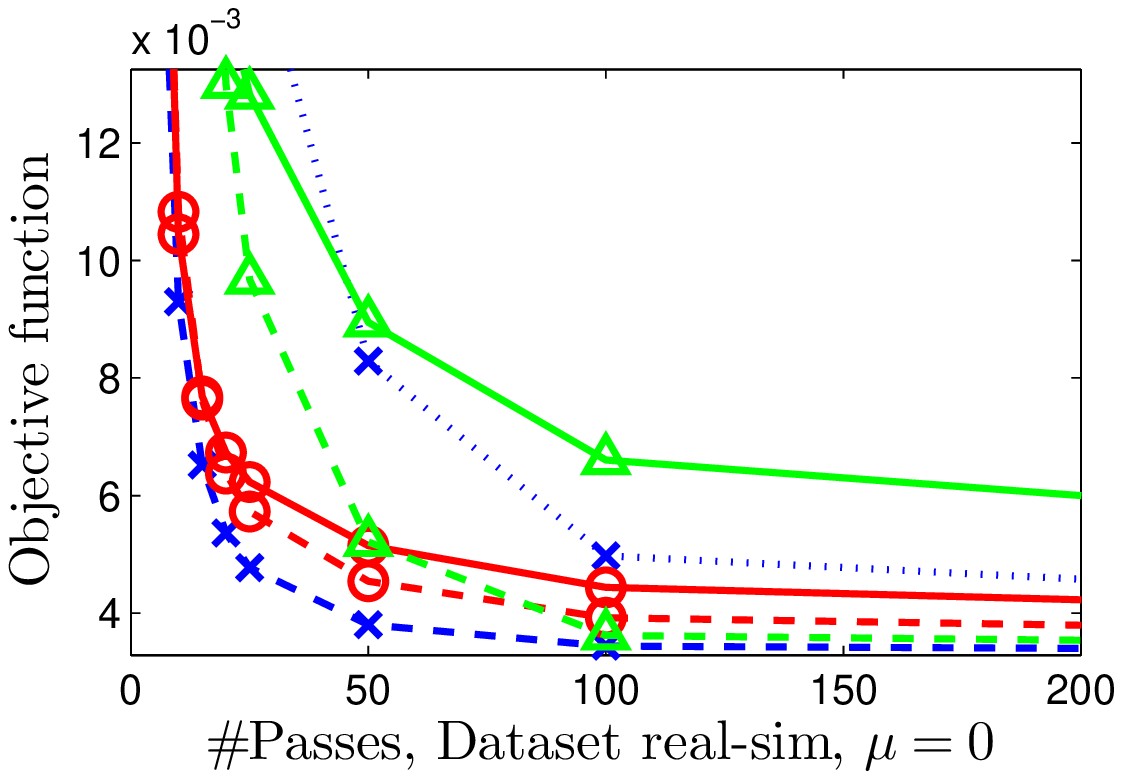}~ 
   \includegraphics[width=0.3\linewidth]{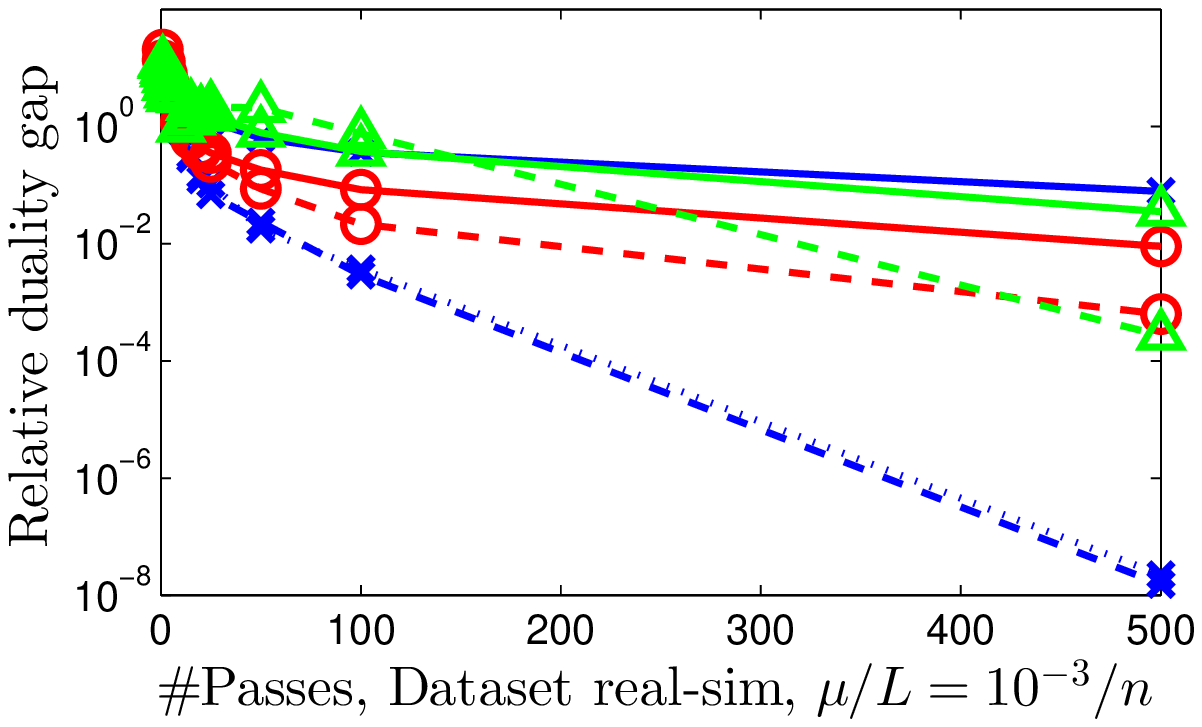} 
   \includegraphics[width=0.3\linewidth]{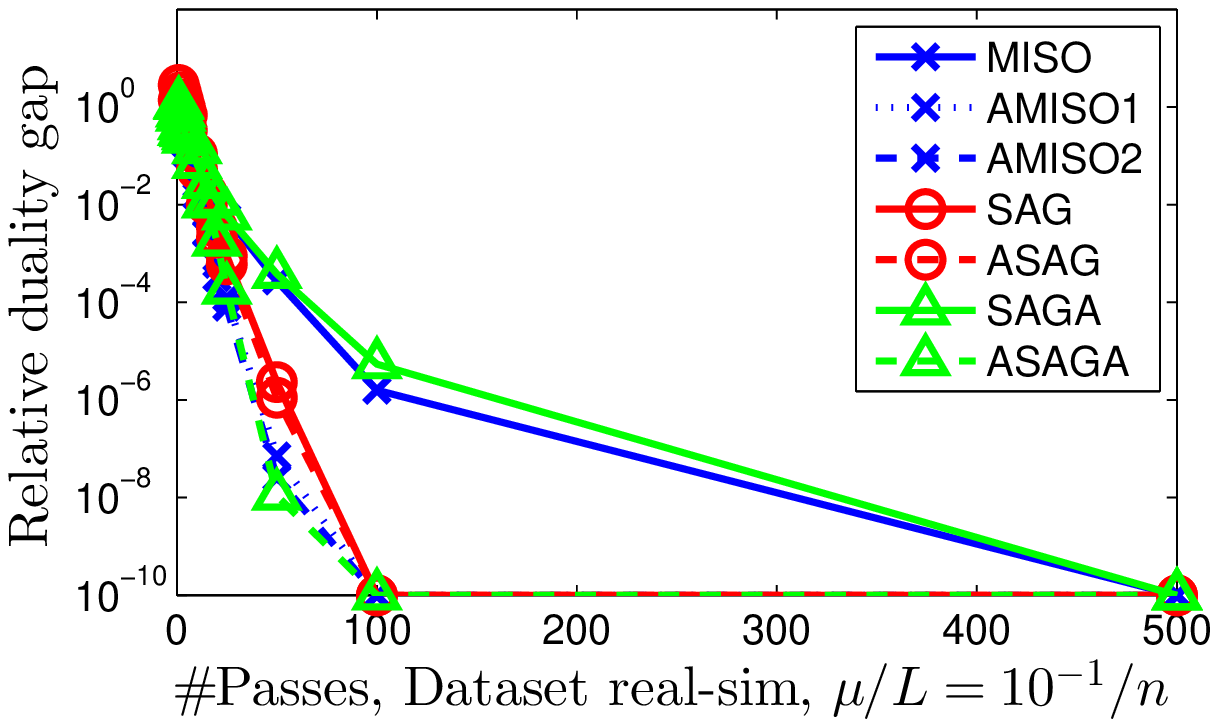} \\
   \includegraphics[width=0.29\linewidth]{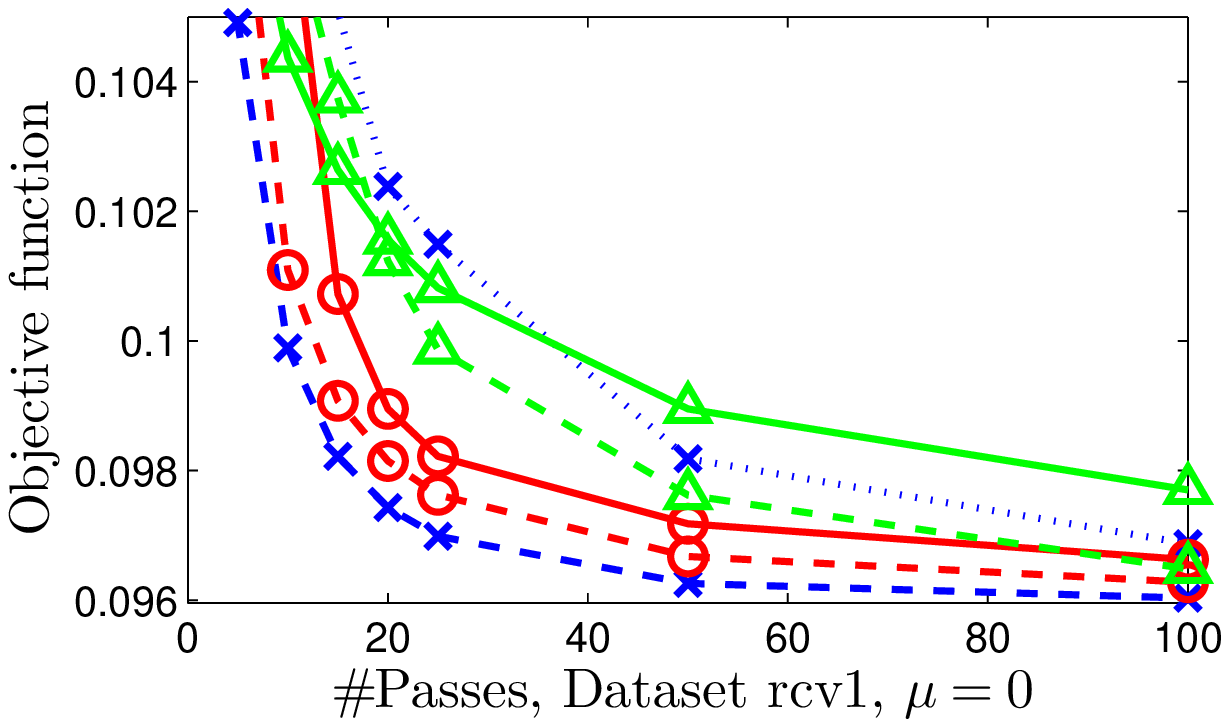} 
   \includegraphics[width=0.3\linewidth]{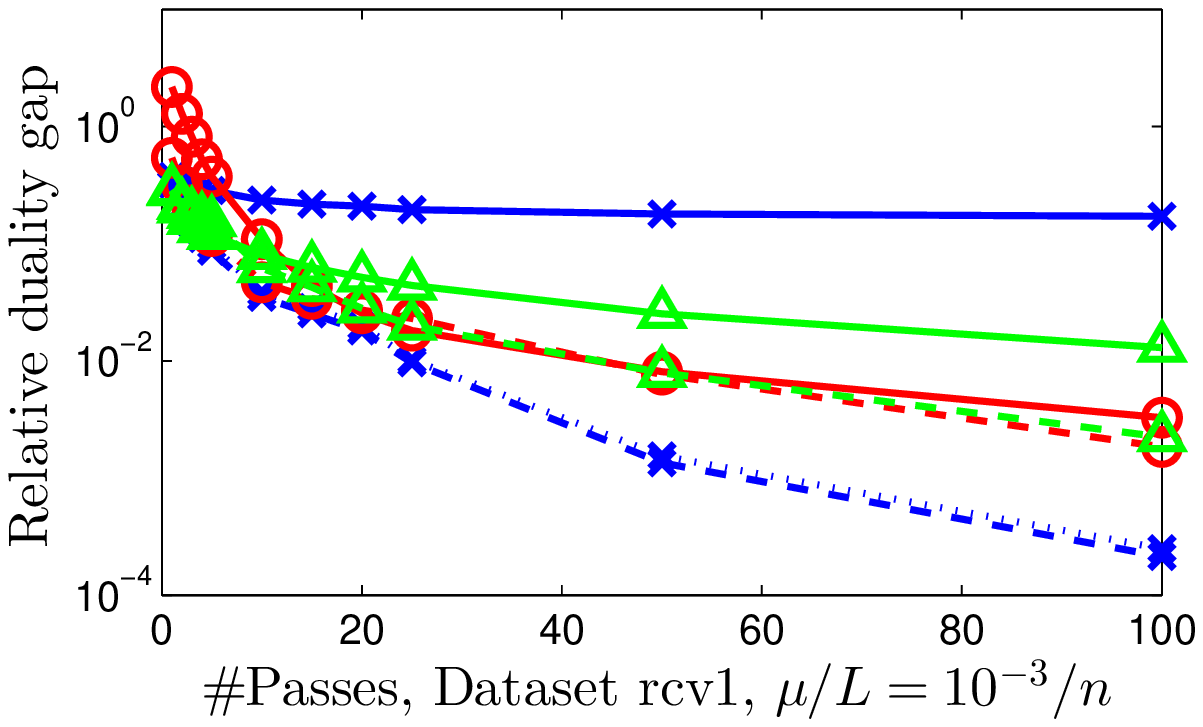} 
   \includegraphics[width=0.3\linewidth]{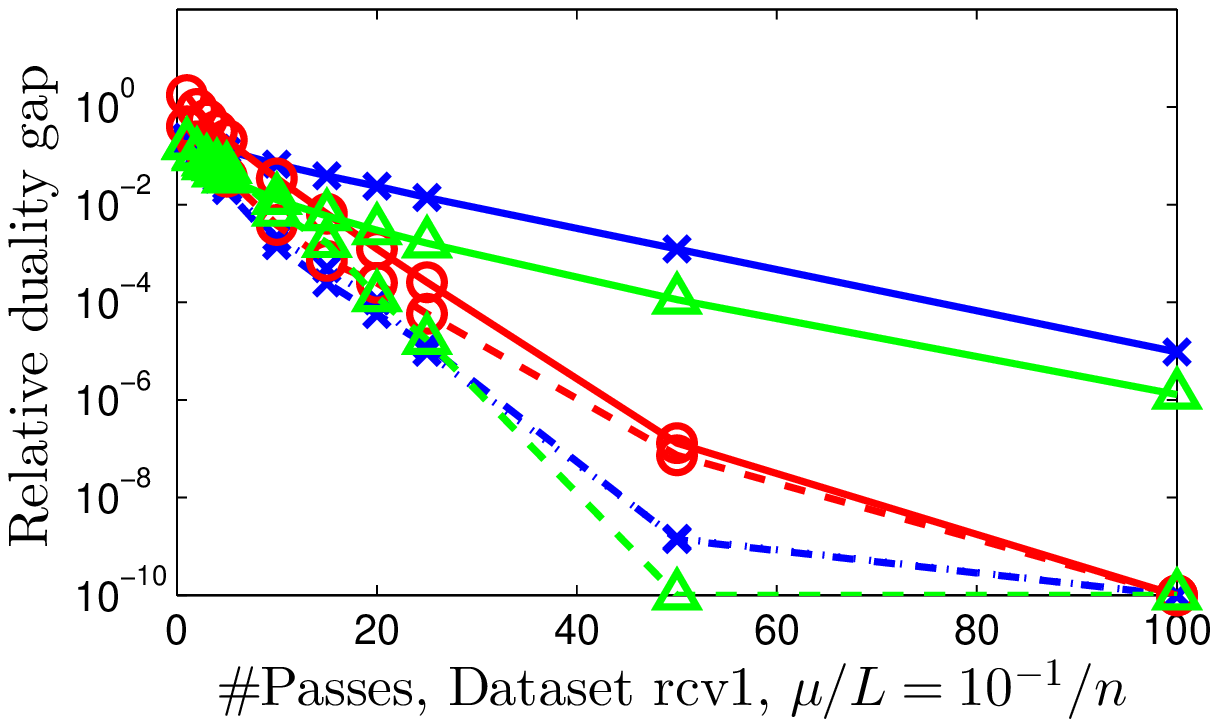} \\
   \includegraphics[width=0.3\linewidth]{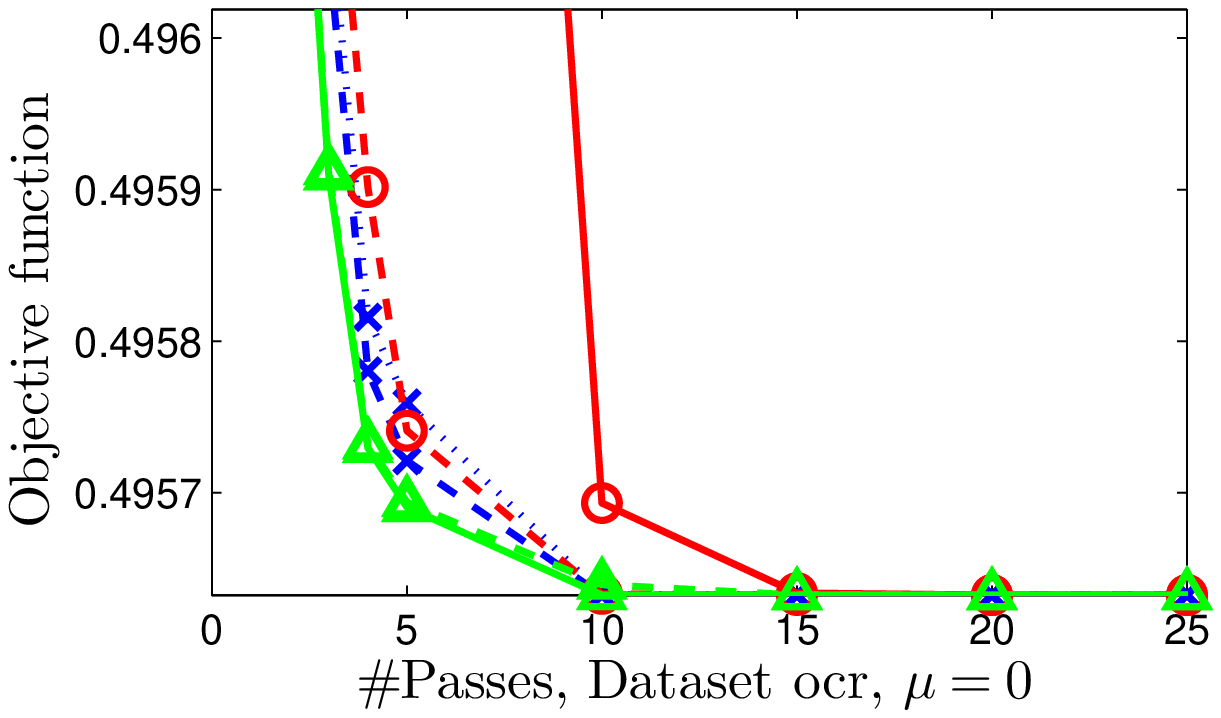} 
   \includegraphics[width=0.3\linewidth]{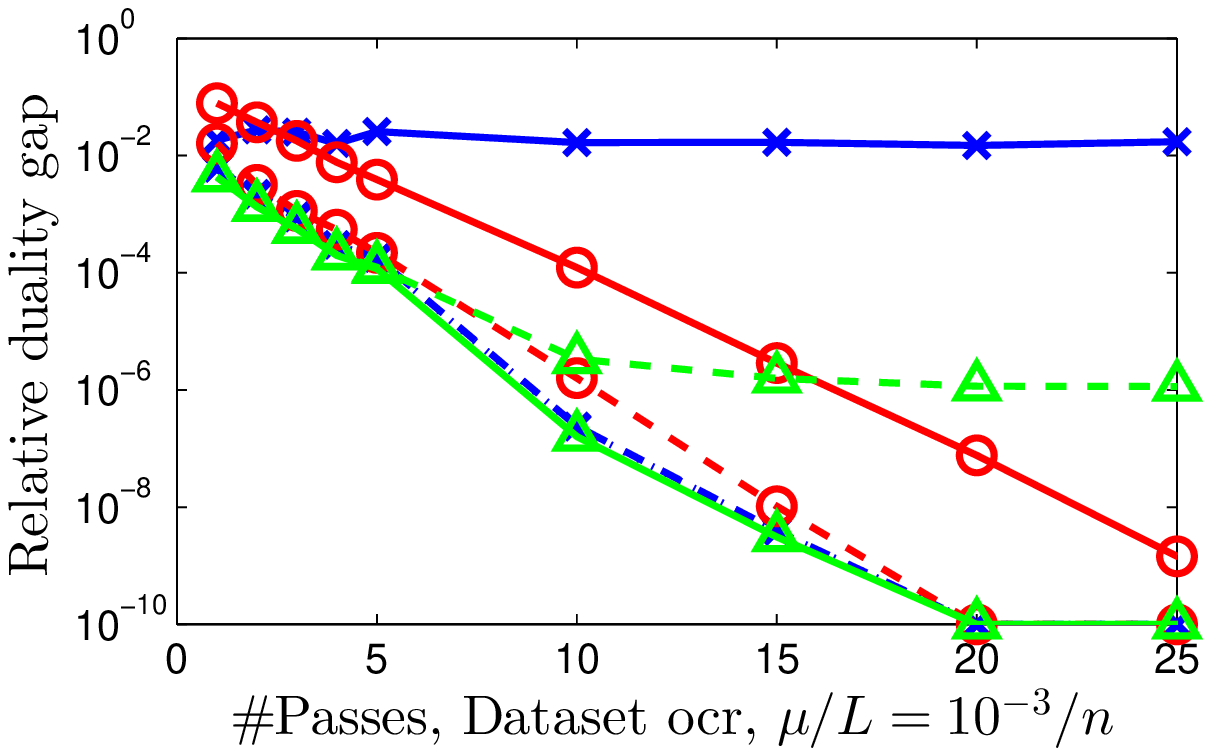} 
   \includegraphics[width=0.3\linewidth]{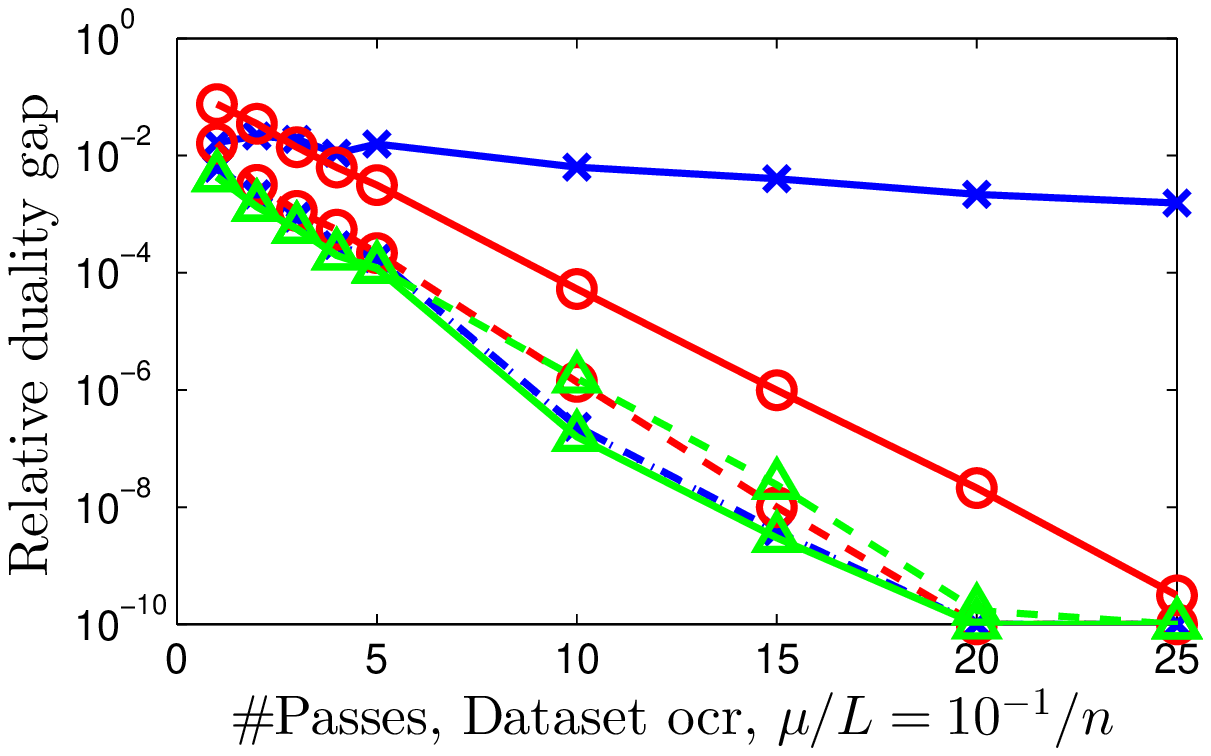} 
   \vs
   \vs
   \caption{Objective function value (or duality gap) for different number of passes performed over each dataset.  
   The legend for all curves is on the top right. AMISO, ASAGA, ASAG refer to the accelerated variants of MISO, SAGA, and SAG, respectively.}\label{table:exp}
   \vs
   \vs
\end{figure}

\vsb
\subsubsection*{Acknowledgments}
\vsb
This work was supported by ANR (MACARON ANR-14-CE23-0003-01),
MSR-Inria joint centre, CNRS-Mastodons program (Titan), 
and NYU Moore-Sloan Data Science Environment.

\newpage
{\small
\bibliographystyle{plain}
\bibliography{bib}

\begin{thebibliography}{10}

\bibitem{agarwal}
A.~Agarwal and L.~Bottou.
\newblock A lower bound for the optimization of finite sums.
\newblock In {\em Proc. International Conference on Machine Learning (ICML)},
  2015.

\bibitem{bach2012optimization}
F.~Bach, R.~Jenatton, J.~Mairal, and G.~Obozinski.
\newblock Optimization with sparsity-inducing penalties.
\newblock {\em Foundations and Trends in Machine Learning}, 4(1):1--106, 2012.

\bibitem{Bauschke:2011}
H.~H. Bauschke and P.~L. Combettes.
\newblock {\em Convex Analysis and Monotone Operator Theory in Hilbert Spaces}.
\newblock Springer, 2011.

\bibitem{fista}
A.~Beck and M.~Teboulle.
\newblock A fast iterative shrinkage-thresholding algorithm for linear inverse
  problems.
\newblock {\em SIAM Journal on Imaging Sciences}, 2(1):183--202, 2009.

\bibitem{bertsekas:2015}
D.~P. Bertsekas.
\newblock {\em Convex Optimization Algorithms}.
\newblock Athena Scientific, 2015.

\bibitem{saga}
A.~J. Defazio, F.~Bach, and S.~Lacoste{-}Julien.
\newblock {SAGA:} {A} fast incremental gradient method with support for
  non-strongly convex composite objectives.
\newblock In {\em Adv. Neural Information Processing Systems (NIPS)}, 2014.

\bibitem{finito}
A.~J. Defazio, T.~S. Caetano, and J.~Domke.
\newblock Finito: A faster, permutable incremental gradient method for big data
  problems.
\newblock In {\em Proc. International Conference on Machine Learning (ICML)},
  2014.

\bibitem{frostig}
R.~Frostig, R.~Ge, S.~M. Kakade, and A.~Sidford.
\newblock Un-regularizing: approximate proximal point algorithms for empirical
  risk minimization.
\newblock In {\em Proc. International Conference on Machine Learning (ICML)},
  2015.

\bibitem{guler:1992}
O.~G\"uler.
\newblock New proximal point algorithms for convex minimization.
\newblock {\em SIAM Journal on Optimization}, 2(4):649--664, 1992.

\bibitem{he2012accelerated}
B.~He and X.~Yuan.
\newblock An accelerated inexact proximal point algorithm for convex
  minimization.
\newblock {\em Journal of Optimization Theory and Applications},
  154(2):536--548, 2012.

\bibitem{hiriart1996convex}
J.-B. Hiriart-Urruty and C.~Lemar{\'e}chal.
\newblock {\em Convex Analysis and Minimization Algorithms I}.
\newblock Springer, 1996.

\bibitem{juditsky_nemirovsky_2012}
A.~Juditsky and A.~Nemirovski.
\newblock First order methods for nonsmooth convex large-scale optimization.
\newblock {\em Optimization for Machine Learning, MIT Press}, 2012.

\bibitem{conjugategradient}
G.~Lan.
\newblock An optimal randomized incremental gradient method.
\newblock {\em arXiv:1507.02000}, 2015.

\bibitem{miso}
J.~Mairal.
\newblock Incremental majorization-minimization optimization with application
  to large-scale machine learning.
\newblock {\em SIAM Journal on Optimization}, 25(2):829--855, 2015.

\bibitem{nemirovski}
A.~Nemirovski, A.~Juditsky, G.~Lan, and A.~Shapiro.
\newblock Robust stochastic approximation approach to stochastic programming.
\newblock {\em {SIAM} Journal on Optimization}, 19(4):1574--1609, 2009.

\bibitem{nesterov1983}
Y.~Nesterov.
\newblock A method of solving a convex programming problem with convergence
  rate {$O$(1/$k^2$)}.
\newblock {\em Soviet Mathematics Doklady}, 27(2):372--376, 1983.

\bibitem{nesterov}
Y.~Nesterov.
\newblock {\em Introductory Lectures on Convex Optimization: A Basic Course}.
\newblock Springer, 2004.

\bibitem{nesterov2012}
Y.~Nesterov.
\newblock Efficiency of coordinate descent methods on huge-scale optimization
  problems.
\newblock {\em SIAM Journal on Optimization}, 22(2):341--362, 2012.

\bibitem{nesterov2013gradient}
Y.~Nesterov.
\newblock Gradient methods for minimizing composite functions.
\newblock {\em Mathematical Programming}, 140(1):125--161, 2013.

\bibitem{Parikh13}
N.~Parikh and S.P. Boyd.
\newblock Proximal algorithms.
\newblock {\em Foundations and Trends in Optimization}, 1(3):123--231, 2014.

\bibitem{richtarik2014}
P.~Richt{\'a}rik and M.~Tak{\'a}{\v{c}}.
\newblock Iteration complexity of randomized block-coordinate descent methods
  for minimizing a composite function.
\newblock {\em Mathematical Programming}, 144(1-2):1--38, 2014.

\bibitem{salzo2012inexact}
S.~Salzo and S.~Villa.
\newblock Inexact and accelerated proximal point algorithms.
\newblock {\em Journal of Convex Analysis}, 19(4):1167--1192, 2012.

\bibitem{proxinexact}
M.~Schmidt, N.~Le Roux, and F.~Bach.
\newblock Convergence rates of inexact proximal-gradient methods for convex
  optimization.
\newblock In {\em Adv. Neural Information Processing Systems (NIPS)}, 2011.

\bibitem{sag}
M.~Schmidt, N.~Le Roux, and F.~Bach.
\newblock Minimizing finite sums with the stochastic average gradient.
\newblock {\em arXiv:1309.2388}, 2013.

\bibitem{sdca}
S.~Shalev-Shwartz and T.~Zhang.
\newblock Proximal stochastic dual coordinate ascent.
\newblock {\em arXiv:1211.2717}, 2012.

\bibitem{accsdca}
S.~Shalev-Shwartz and T.~Zhang.
\newblock Accelerated proximal stochastic dual coordinate ascent for
  regularized loss minimization.
\newblock {\em Mathematical Programming}, 2015.

\bibitem{proxsvrg}
L.~Xiao and T.~Zhang.
\newblock A proximal stochastic gradient method with progressive variance
  reduction.
\newblock {\em {SIAM} Journal on Optimization}, 24(4):2057--2075, 2014.

\bibitem{zhangxiao}
Y.~Zhang and L.~Xiao.
\newblock Stochastic primal-dual coordinate method for regularized empirical
  risk minimization.
\newblock In {\em Proc. International Conference on Machine Learning (ICML)},
  2015.

\end{thebibliography}
\nocite{conjugategradient}
}

\newpage
\appendix
In this appendix, Section~\ref{sec:estimatesequence} is devoted to the
construction of an object called \emph{estimate sequence}, originally
introduced by Nesterov~(see \cite{nesterov}), and introduce extensions
to deal with inexact minimization. This section contains a generic 
convergence result that will be used to prove the main theorems and
propositions of the paper in Section~\ref{sec:proofs}. Then, Section~\ref{sec:accel-factor-calculs}
is devoted to the computation of global convergence rates of accelerated algorithms,
Section~\ref{appendix:miso} presents in details the proximal MISO algorithm, and {Section~\ref{appendix:exp} gives some implementation details of the experiments}.

\section{Construction of the Approximate Estimate Sequence}\label{sec:estimatesequence}
The estimate sequence is a generic tool introduced by Nesterov for proving the
convergence of accelerated gradient-based algorithms.  We start by recalling
the definition given in~\cite{nesterov}.
\begin{defi}[\bfseries Estimate Sequence~\cite{nesterov}]~\label{def:estimate}\newline
   A pair of sequences $(\varphi_k)_{k \geq 0}$ and $(\lambda_k)_{k \geq 0}$,
   with $\lambda_k \geq 0$ and $\varphi_k: \Real^p \to \Real$, is called an \emph{estimate sequence} of function~$F$ if
   \begin{displaymath}
      \lambda_k \to 0,
   \end{displaymath}
   and for any~$x$ in~$\Real^p$ and all $k \geq 0$, we have
   \begin{displaymath}
      \varphi_k(x) \leq (1-\lambda_k)F(x) + \lambda_k \varphi_0(x).
   \end{displaymath}
\end{defi}
Estimate sequences are used for proving convergence rates thanks to the following lemma
\begin{lemma}[\bfseries Lemma 2.2.1 from~\cite{nesterov}]~\label{lemma:nest}\newline
   If for some sequence $(x_k)_{k \geq 0}$ we have
   \begin{displaymath}
      F(x_k) \leq \varphi_k^\star \defin \min_{x \in \Real^p} \varphi_k(x),
   \end{displaymath}
   for an estimate sequence~$(\varphi_k)_{k \geq 0}$ of~$F$, then
   \begin{displaymath}
      F(x_k) - F^\star \leq \lambda_k (\varphi_0(x^\star) - F^\star),
   \end{displaymath}
   where~$x^\star$ is a minimizer of~$F$.
\end{lemma}
The rate of convergence of~$F(x_k)$ is thus directly related to the convergence
rate of $\lambda_k$. Constructing estimate sequences is thus appealing, even
though finding the most appropriate one is not trivial for the
catalyst algorithm because of the approximate minimization of~$G_k$
in~(\ref{eq:approx}). In a nutshell, the main steps of our convergence analysis are
\begin{enumerate}
   \item define an ``approximate'' estimate sequence for~$F$ corresponding to
      Algorithm~\ref{alg:catalyst}---that is, finding a function~$\varphi$
      that almost satisfies Definition~\ref{def:estimate} up to the
      approximation errors~$\varepsilon_k$ made in~(\ref{eq:approx}) when
      minimizing~$G_k$, and control the way these errors sum up together.
   \item extend Lemma~\ref{lemma:nest} to deal with the approximation
      errors~$\varepsilon_k$ to derive a generic convergence rate for the sequence~$(x_k)_{k \geq 0}$.
\end{enumerate}
This is also the strategy proposed by G\"uler in~\cite{guler:1992} for his
inexact accelerated proximal point algorithm, which essentially differs from
ours in its stopping criterion. The estimate sequence we choose is also
different and leads to a more rigorous convergence proof. Specifically, we
prove in this section the following theorem:
\begin{thm}[\bfseries Convergence Result Derived from an Approximate Estimate Sequence]~\label{thm}\newline
Let us denote 
\begin{equation}
   \lambda_k = \prod_{i=0}^{k-1} (1- \alpha_i),  \label{eq:lambdak}
\end{equation}
where the~$\alpha_i$'s are defined in Algorithm~\ref{alg:catalyst}. Then, the sequence $(x_{k})_{k \geq 0}$ satisfies 
\begin{equation}\label{F}
F(x_{k}) -F^* \leqslant \lambda_{k} \left ( \sqrt{S_k}+ 2\sum_{i=1}^{k} \sqrt {\frac{\epsilon_i }{\lambda_{i}}} \right )^2,
\end{equation}
where~$F^\star$ is the minimum value of~$F$ and  
\begin{equation}
   S_k =  F(x_0)  -F^*+ \frac{\gamma_0}{2} \Vert x_0 - x^*\Vert^2+
\sum_{i=1}^{k}\frac{\epsilon_i}{\lambda_{i}}~~~~\text{where}~~~~~\gamma_0 =
\frac{\alpha_0\left((\kappa+\mu)\alpha_0-\mu\right)}{1-\alpha_0}, \label{eq:gamma0}
\end{equation}
where~$x^\star$ is a minimizer of~$F$.
\end{thm}
Then, the theorem will be used with the following lemma from~\cite{nesterov} to
control the convergence rate of the sequence~$(\lambda_k)_{k \geq 0}$, whose
definition follows the classical use of estimate sequences~\cite{nesterov}.
This will provide us convergence rates both for the strongly convex and
non-strongly convex cases.
\begin{lemma}[\bfseries Lemma 2.2.4 from~\cite{nesterov}]~\label{lemma:lambda}\newline
   If in the quantity~$\gamma_0$ defined in~(\ref{eq:gamma0}) satisfies $\gamma_0 \geqslant \mu$, then the sequence~$(\lambda_k)_{k \geq 0}$ from~(\ref{eq:lambdak}) satisfies
\begin{equation} \label{lambda}
\lambda_k \leqslant \min \left \{  \left (1-\sqrt{q} \right )^k, \frac{4}{ \left (2 + k \sqrt{\frac{\gamma_0}{\kappa+\mu}} \right )^2}   \right \}.
\end{equation}
\end{lemma}
We may now move to the proof of the theorem.
\subsection{Proof of Theorem~\ref{thm}}
   The first step is to construct an estimate sequence is typically 
   to find a sequence of lower bounds of~$F$. By calling~$x_k^*$ the
   minimizer of~$G_k$, the following one is used in~\cite{guler:1992}:
   \begin{lemma}[\bfseries Lower Bound for~$F$ near $x_k^*$]~\newline
For all $x$ in $\R^p$, 
\begin{equation}\label{xk*}
F(x) \geqslant F(x_k^*) + \langle \kappa (y_{k-1} - x_k^*), x- x_k^* \rangle + \frac{\mu}{2} \Vert x- x_k^* \Vert^2. 
\end{equation}
\end{lemma}
\begin{proof}
By strong convexity, $G_k(x) \geqslant G_k(x_k^*) + \frac{\kappa+\mu}{2} \Vert x - x_{k}^* \Vert^2$, which is equivalent to
$$ F(x) + \frac{\kappa}{2} \Vert x - y_k \Vert^2 \geqslant F(x_k^*) + \frac{\kappa}{2} \Vert x_k^* - y_{k-1} \Vert^2 + \frac{\kappa+\mu}{2} \Vert x-x_k^* \Vert^2.$$
After developing the quadratic terms, we directly obtain (\ref{xk*}).
\end{proof}

Unfortunately, the exact value $x_k^*$ is unknown in practice and the estimate
sequence of~\cite{guler:1992} yields in fact an algorithm where the definition
of the anchor point~$y_k$ involves the unknown quantity $x_k^*$ instead of the
approximate solutions $x_k$ and~$x_{k-1}$ as in~(\ref{y_k}), as also noted by
others~\cite{salzo2012inexact}. To obtain a rigorous proof of convergence for
Algorithm~\ref{alg:catalyst}, it is thus necessary to refine the analysis
of~\cite{guler:1992}. To that effect, we construct below a sequence of functions 
that approximately satisfies the definition of estimate sequences. Essentially, 
we replace in~(\ref{xk*}) the quantity~$x_k^*$ by~$x_k$ to obtain an approximate
lower bound, and control the error by using the condition $G_k(x_{k}) - G_k(x_k^*) \leqslant \epsilon_k$. 
This leads us to the following construction:
\begin{enumerate}
\item $\phi_0(x) = F(x_0) + \frac{\gamma_0}{2} \Vert x - x_0 \Vert^2$;
\item For $k \geqslant 1$, we set 
$$ \phi_{k}(x) = (1- \alpha_{k-1}) \phi_{k-1}(x) + \alpha_{k-1} [ F(x_{k}) + \langle \kappa (y_{k-1} - x_{k}), x-x_{k} \rangle+ \frac{\mu}{2} \Vert x - x_{k} \Vert^2 ],  $$
\end{enumerate}
where the value of~$\gamma_0$, given in~(\ref{eq:gamma0}) will be explained later.  Note that if one
replaces~$x_k$ by~$x_k^*$ in the above construction, it is easy to show that
$(\phi_k)_{k \geq 0}$ would be exactly an estimate sequence for~$F$ with the
relation~$\lambda_k$ given in~(\ref{eq:lambdak}).

Before extending Lemma~\ref{lemma:nest} to deal with the approximate sequence and
conclude the proof of the theorem, we need to characterize a few properties of
the sequence~$(\phi_k)_{k \geq 0}$. In particular, the functions~$\phi_k$ are quadratic
and admit a canonical form:
\begin{lemma}[\bfseries Canonical Form of the Functions~$\phi_k$]~\label{estimate}\newline
For all $k \geq 0$, $\phi_k$ can be written in the canonical form 
$$\phi_{k}(x) = \phi_{k}^* + \frac{\gamma_{k}}{2} \Vert x - v_{k} \Vert^2,$$ 
where the sequences $(\gamma_k)_{k \geq 0}$, $(v_k)_{k \geq 0}$, and $(\phi_k^*)_{ k \geq 0}$ are defined as follows
\begin{align} 
   \gamma_{k} = \,& (1- \alpha_{k-1}) \gamma_{k-1} + \alpha_{k-1} \mu, \label{eq:gammak} \\
   v_{k} = \, & \frac{1}{\gamma_{k}} \left((1-\alpha_{k-1}) \gamma_{k-1} v_{k-1} + \alpha_{k-1} \mu x_{k} - \alpha_{k-1} \kappa (y_{k-1} -x_{k})\right),\label{eq:vk} \\
\phi_{k}^* = \, & (1-\alpha_{k-1}) \phi_{k-1}^* + \alpha_{k-1} F(x_{k}) - \frac{\alpha_{k-1}^2}{2\gamma_{k}} \Vert \kappa (y_{k-1} -x_{k}) \Vert^2 \nonumber \\ 
                & + \frac{\alpha_{k-1} (1-\alpha_{k-1}) \gamma_{k-1}}{\gamma_{k}} \left ( \frac{\mu}{2} \Vert x_{k} - v_{k-1} \Vert^2 + \langle \kappa (y_{k-1} -x_{k}), v_{k-1} - x_{k} \rangle \right ), \label{eq:phikstar}
\end{align}
\end{lemma}

\begin{proof}
   We have for all~$k \geq 1$ and all~$x$ in $\Real^p$,
\begin{equation}\label{estimate sequence}
   \begin{split}
 \phi_{k}(x)
 = \,\,&(1-\alpha_{k-1}) \left(\phi_{k-1}^* + \frac{\gamma_{k-1}}{2} \Vert x - v_{k-1} \Vert^2\right) \\
       & ~~~~~~+ \alpha_{k-1} \left(F(x_{k}) + \langle \kappa (y_{k-1} - x_{k}), x-x_{k} \rangle+ \frac{\mu}{2} \Vert x - x_{k} \Vert^2 \right)  \\
 = \,\, &\phi_{k}^* + \frac{\gamma_{k}}{2} \Vert x - v_{k} \Vert^2.
 \end{split}
\end{equation}
Differentiate twice the relations (\ref{estimate sequence}) gives us directly~(\ref{eq:gammak}).
Since $v_k$ minimizes~$\phi_k$, the optimality condition $\nabla \phi_{k} (v_k) = 0 $ gives
\begin{align*}
   (1-\alpha_{k-1}) \gamma_{k-1} (v_k-v_{k-1}) + \alpha_{k-1} \left( \kappa (y_{k-1}-x_{k}) + \mu (v_k-x_{k}) \right) = 0, 
\end{align*}
and then we obtain~(\ref{eq:vk}).
Finally, apply $x= x_k$ to (\ref{estimate sequence}), which yields
\begin{displaymath}
   \phi_{k}(x_{k}) =  (1- \alpha_{k-1}) \left(\phi_{k-1}^*+ \frac{\gamma_{k-1}}{2} \Vert x_{k} - v_{k-1} \Vert^2  \right) + \alpha_{k-1} F(x_{k}) = \phi_{k}^* + \frac{\gamma_{k}}{2} \Vert x_{k} - v_{k} \Vert^2.
\end{displaymath}
Consequently, 
\begin{equation} \label{phi*}
\phi_{k}^* =  (1- \alpha_{k-1}) \phi_{k-1}^*  + \alpha_{k-1} F(x_{k})+ (1-\alpha_{k-1})\frac{\gamma_{k-1}}{2} \Vert x_{k} -v_{k-1} \Vert^2 - \frac{\gamma_{k}}{2} \Vert x_{k} - v_{k} \Vert^2 
\end{equation}
Using the expression of $v_k$ from~(\ref{eq:vk}), we have 
$$ v_{k} - x_{k}  = \frac{1}{\gamma_{k}} \left((1-\alpha_{k-1}) \gamma_{k-1} (v_{k-1} - x_{k}) - \alpha_{k-1} \kappa (y_{k-1} - x_{k})\right).$$
Therefore
\begin{equation*}
   \begin{split}
      \frac{\gamma_{k}}{2} \Vert x_{k} - v_{k} \Vert^2  & =  \frac{(1-\alpha_{k-1})^2 \gamma_{k-1}^2}{2\gamma_{k}} \Vert x_{k} -v_{k-1} \Vert^2 \\ 
                                                        & - \frac{(1-\alpha_{k-1})\alpha_{k-1} \gamma_{k-1}}{\gamma_{k}} \langle v_{k-1}- x_{k}, \kappa(y_{k-1}-x_{k}) \rangle + \frac{\alpha_{k-1}^2}{2\gamma_{k}} \Vert \kappa(y_{k-1}-x_{k})\Vert^2.
   \end{split}
\end{equation*}
It remains to plug this relation into (\ref{phi*}), use once~(\ref{eq:gammak}), and we obtain the formula~(\ref{eq:phikstar}) for $\phi_{k}^*$.
\end{proof}
We may now start analyzing the errors~$\varepsilon_k$ to control how far is the
sequence~$(\phi_k)_{k \geq 0}$ from an exact estimate sequence.
For that, we need to understand the effect of replacing~$x_k^*$ by~$x_k$ in the lower bound~(\ref{xk*}). 
The following lemma will be useful for that purpose.
\begin{lemma}[\bfseries Controlling the Approximate Lower Bound of~$F$]\label{lemma1.1}~\newline
   If $G_k(x_{k}) - G_k(x_k^*) \leqslant \epsilon_k$, then for all $x$ in $\R^p$,
\begin{equation} \label{eq3}
F(x) \geqslant F(x_{k}) + \langle \kappa (y_{k-1} - x_{k}), x - x_{k} \rangle + \frac{\mu}{2} \Vert x-x_{k} \Vert^2 + (\kappa+\mu) \langle x_{k} - x_{k}^*, x-x_{k} \rangle - \epsilon_k.
\end{equation}
\end{lemma}
\begin{proof}
   By strong convexity, for
   all $x$ in~$\Real^p$,
$$ G_k(x) \geqslant G_k^* + \frac{\kappa+\mu}{2} \Vert x - x_{k}^* \Vert^2,$$
where~$G_k^\star$ is the minimum value of~$G_k$.
Replacing~$G_k$ by its definition~(\ref{eq:approx}) gives 
\begin{align*}
F(x)  \geqslant & \,\,  G_k^* + \frac{\kappa+\mu}{2} \Vert x - x_{k}^* \Vert^2 - \frac{\kappa}{2} \Vert x - y_{k-1} \Vert^2 \\
	= & \,\,G_k(x_{k}) +  (G_k^*-G_k(x_{k})) + \frac{\kappa + \mu}{2}  \Vert x - x_{k}^* \Vert^2 - \frac{\kappa}{2} \Vert x - y_{k-1} \Vert^2 \\
	\geqslant & \,\,G_k(x_{k})  - \epsilon_k + \frac{\kappa + \mu}{2} \Vert (x - x_{k}) + (x_{k} -x_{k}^*) \Vert^2  - \frac{\kappa}{2} \Vert x - y_{k-1} \Vert^2 \\
	\geq & \,\, G_k(x_{k})  - \epsilon_k + \frac{\kappa + \mu}{2} \Vert x - x_{k} \Vert^2 - \frac{\kappa}{2} \Vert x - y_{k-1} \Vert^2 + (\kappa+\mu) \langle x_{k} - x_{k}^*, x-x_{k} \rangle. 
\end{align*}
We conclude by noting that
\begin{displaymath}
   \begin{split}
      G_k(x_{k}) \!+\! \frac{\kappa}{2} \Vert x - x_{k} \Vert^2 \!-\! \frac{\kappa}{2} \Vert x - y_{k-1} \Vert^2 & = \!
  F(x_{k}) \!+ \!\frac{\kappa}{2} \Vert x_{k} -y_{k-1} \Vert^2 \!+ \!\frac{\kappa}{2} \Vert x - x_{k} \Vert^2 \!- \!\frac{\kappa}{2} \Vert x - y_{k-1} \Vert^2 \\
  & =  F(x_{k}) + \langle \kappa (y_{k-1} - x_{k}), x - x_{k} \rangle.
 \end{split}
\end{displaymath}
\end{proof}
We can now show that Algorithm~\ref{alg:catalyst} generates iterates~$(x_k)_{k \geq 0}$ that
approximately satisfy the condition of Lemma~\ref{lemma:nest} from Nesterov~\cite{nesterov}.
\begin{lemma}[\bfseries Relation between~$(\phi_k)_{k \geq 0}$ and Algorithm~\ref{alg:catalyst}]\label{xik}~\newline
   Let $\phi_k$ be the estimate sequence constructed above. Then, Algorithm~\ref{alg:catalyst} generates iterates~$(x_k)_{k \geq 0}$ such that
$$ F(x_{k}) \leqslant \phi_{k}^* + \xi_{k}, $$
where the sequence $(\xi_k)_{k \geq 0}$ is defined by $\xi_0 = 0$ and
$$ \xi_{k} = (1-\alpha_{k-1}) (\xi_{k-1} + \epsilon_k - (\kappa+\mu) \langle x_{k} - x_{k}^*, x_{k-1}-x_{k} \rangle  ). $$
\end{lemma}

\begin{proof}
We proceed by induction. For $k=0$, $\phi_0^* = F(x_0)$ and $\xi_0 =0$. \\ 
Assume now that $ F(x_{k-1}) \leqslant \phi_{k-1}^* + \xi_{k-1} $. 
Then, 
\begin{equation*}
   \begin{split}
      \phi_{k-1}^\star  & \geq F(x_{k-1}) - \xi_{k-1} \\
                        & \geq F(x_{k}) + \langle \kappa (y_{k-1} - x_{k}), x_{k-1} - x_{k} \rangle + (\kappa+\mu) \langle x_{k} - x_{k}^*, x_{k-1}-x_{k} \rangle - \epsilon_k - \xi_{k-1} \\
      & = F(x_{k}) + \langle \kappa (y_{k-1} - x_{k}), x_{k-1} - x_{k} \rangle - \xi_k / (1-\alpha_{k-1}),
   \end{split}
\end{equation*}
where the second inequality is due to (\ref{eq3}).
By Lemma~\ref{estimate}, we now have,
\begin{equation*}
\begin{split}
   \phi_{k}^*  = &\,\, (1-\alpha_{k-1}) \phi_{k-1}^* + \alpha_{k-1} F(x_{k}) - \frac{\alpha_{k-1}^2}{2\gamma_{k}} \Vert \kappa (y_{k-1} -x_{k}) \Vert^2 \\
& ~~~~+ \frac{\alpha_{k-1} (1-\alpha_{k-1}) \gamma_{k-1}}{\gamma_{k}} \left ( \frac{\mu}{2} \Vert x_{k} - v_{k-1} \Vert^2 + \langle \kappa (y_{k-1} -x_{k}), v_{k-1} - x_{k} \rangle \right ) \\
   \geqslant & \,\,(1-\alpha_{k-1}) \left ( F(x_{k}) + \langle \kappa (y_{k-1} - x_{k}), x_{k-1} - x_{k} \rangle\right ) - \xi_k + \alpha_{k-1} F(x_{k})\\
 & ~~~~ - \frac{\alpha_{k-1}^2}{2\gamma_{k}} \Vert \kappa (y_{k-1} -x_{k}) \Vert^2 + \frac{\alpha_{k-1} (1-\alpha_{k-1}) \gamma_{k-1}}{\gamma_{k}}  \langle \kappa (y_{k-1}-x_{k}), v_{k-1} - x_{k} \rangle. \\
 =&  \,\, 	F(x_{k}) + (1-\alpha_{k-1}) \langle \kappa (y_{k-1} - x_{k}), x_{k-1} - x_{k}+ \frac{\alpha_{k-1} \gamma_{k-1}}{\gamma_{k}} (v_{k-1} - x_{k}) \rangle \\ 
  & ~~~~- \frac{\alpha_{k-1}^2}{2\gamma_{k}} \Vert \kappa (y_{k-1} -x_{k}) \Vert^2 -\xi_k \\
   =& \,\,  F(x_{k}) + (1-\alpha_{k-1}) \langle \kappa (y_{k-1} - x_{k}), x_{k-1}- y_{k-1}+\frac{\alpha_{k-1} \gamma_{k-1}}{\gamma_{k}} ( v_{k-1} - y_{k-1}) \rangle \\
&~~~~+ \left (1  -\frac{ (\kappa+2\mu) \alpha_{k-1}^2}{2\gamma_{k}} \right) \kappa \Vert (y_{k-1} -x_{k}) \Vert^2 -\xi_k. 
   \end{split}
\end{equation*}
We now need to show that the choice of the sequences~$(\alpha_k)_{k \geq 0}$ and~$(y_k)_{k\geq 0}$ will cancel all the terms involving $y_{k-1}-x_k$. In other words,
we want to show that 
\begin{equation} \label{recurrence}
 x_{k-1}- y_{k-1}+\frac{\alpha_{k-1} \gamma_{k-1}}{\gamma_{k}} ( v_{k-1} - y_{k-1}) = 0,
\end{equation} 
and we want to show that
\begin{equation}\label{alphak2}
1-  (\kappa + \mu)\frac{\alpha_{k-1}^2}{\gamma_{k}} = 0,
\end{equation}
which will be sufficient to conclude that $\phi_{k}^* + \xi_{k} \geqslant F(x_{k})$. 
The relation~(\ref{alphak2}) can be obtained from the definition of~$\alpha_k$ in~(\ref{y_k}) and the form of~$\gamma_k$ given in~(\ref{eq:gammak}). We have
indeed from~(\ref{y_k}) that
\begin{align*}
   (\kappa+\mu) \alpha_{k}^2 = (1-\alpha_{k}) (\kappa+\mu)\alpha_{k-1}^2 + \alpha_{k} \mu.
\end{align*}
Then, the quantity $(\kappa+\mu) \alpha_{k}^2$ follows the same recursion as $\gamma_{k+1}$ in~(\ref{eq:gammak}). Moreover, we have 
\begin{displaymath}
   \gamma_1 = (1-\alpha_0)\gamma_0 + \mu \alpha_0 = (\kappa+\mu)\alpha_0^2,
\end{displaymath}
from the definition of~$\gamma_0$ in~(\ref{eq:gamma0}). We can then conclude by induction that 
$\gamma_{k+1}=(\kappa+\mu) \alpha_{k}^2$ for all~$k \geq 0$ and~(\ref{alphak2}) is satisfied.

To prove~(\ref{recurrence}), we assume that~$y_{k-1}$ is chosen such that~(\ref{recurrence}) is satisfied,
and show that it is equivalent to defining~$y_k$ as in~(\ref{y_k}).
By lemma~\ref{estimate},
 \begin{align}
 v_{k} & =  \frac{1}{\gamma_{k}} \left((1-\alpha_{k-1}) \gamma_{k-1} v_{k-1} + \alpha_{k-1} \mu x_{k} - \alpha_{k-1} \kappa (y_{k-1} -x_{k})\right)  \nonumber \\
 	    & =  \frac{1}{\gamma_{k}} \left (\frac{(1-\alpha_{k-1})}{\alpha_{k-1}} ((\gamma_{k}+ \alpha_{k-1} \gamma_{k-1})y_{k-1} - \gamma_{k} x_{k-1})   + \alpha_{k-1} \mu x_{k} - \alpha_{k-1} \kappa (y_{k-1} -x_{k}) \right ) \nonumber \\
 	    & = \frac{1}{\gamma_{k}} \left (\frac{(1-\alpha_{k-1})}{\alpha_{k-1}} ((\gamma_{k-1} + \alpha_{k-1} \mu)y_{k-1} - \gamma_{k} x_{k-1})   + \alpha_{k-1} (\mu+\kappa) x_{k} - \alpha_{k-1} \kappa y_{k-1}  \right ) \nonumber \\
 	    & = \frac{1}{\gamma_{k}} \left (\frac{1}{\alpha_{k-1}} (\gamma_{k} - \mu \alpha_{k-1}^2 )y_{k-1} - \frac{(1-\alpha_{k-1})}{\alpha_{k-1}} \gamma_{k} x_{k-1}   + \frac{\gamma_{k}}{\alpha_{k-1}} x_{k} - \alpha_{k-1} \kappa y_{k-1}  \right ) \nonumber \\
      & = \frac{1}{\alpha_{k-1}} (x_{k} - (1-\alpha_{k-1})x_{k-1}), \label{vkalt}
 \end{align}
 As a result, using (\ref{recurrence}) by replacing $k-1$ by $k$ yields
 $$  y_{k} =  x_{k} + \frac{\alpha_{k-1} (1-\alpha_{k-1})}{\alpha_{k-1}^2 + \alpha_{k}} (x_{k} - x_{k-1}),$$
 and we obtain the original equivalent definition of~(\ref{y_k}).
 This concludes the proof.
\end{proof}

With this lemma in hand, we introduce the following proposition, which brings us almost to 
Theorem~\ref{thm}, which we want to prove.
\begin{prop}[\bfseries Auxiliary Proposition for Theorem~\ref{thm}]~\newline
   Let us consider the sequence~$(\lambda_k)_{k \geq 0}$ defined in~(\ref{eq:lambdak}).
   Then, the sequence $(x_{k})_{k \geq 0}$ satisfies 
\begin{equation*} 
 \frac{1}{\lambda_{k}} (F(x_{k}) - F^* + \frac{\gamma_{k}}{2} \Vert x^* - v_{k} \Vert^2) \leqslant \phi_{0} (x^*) -F^*+ \sum_{i=1}^{k}\frac{\epsilon_i}{\lambda_{i}} + \sum_{i=1}^{k} \frac{ \sqrt{2 \epsilon_i \gamma_{i} }} {\lambda_{i}} \Vert x^* - v_{i} \Vert,
 \end{equation*}
 where~$x^\star$ is a minimizer of~$F$ and~$F^\star$ its minimum value.
\end{prop}
\begin{proof}
By the definition of the function~$\phi_k$, we have
\begin{equation*}
   \begin{split}
      \phi_{k} (x^*) & = (1-\alpha_{k-1}) \phi_{k-1} (x^*) + \alpha_{k-1} [F(x_{k}) +  \langle \kappa (y_{k-1} - x_{k}), x^* - x_{k} \rangle + \frac{\mu}{2} \Vert x^* - x_{k} \Vert^2 ] \\
       &\leqslant (1-\alpha_{k-1}) \phi_{k-1} (x^*) + \alpha_{k-1} [F(x^*)+ \epsilon_k  - (\kappa+ \mu) \langle x_{k} - x_{k}^*, x^*-x_{k} \rangle],
   \end{split}
\end{equation*}
where the inequality comes from (\ref{eq3}).
Therefore, by using the definition of~$\xi_k$ in Lemma~\ref{xik},
\begin{eqnarray*}
		&&\phi_{k} (x^*) + \xi_{k} - F^*\\
		&\leqslant& (1-\alpha_{k-1}) (\phi_{k-1} (x^*) +\xi_{k-1} -F^*)+ \epsilon_k -(\kappa+\mu) \langle x_{k} - x_{k}^*, (1-\alpha_{k-1}) x_{k-1} + \alpha_{k-1} x^* - x_{k} \rangle \\
		&=& (1-\alpha_{k-1}) (\phi_{k-1} (x^*) +\xi_{k-1} -F^*)+ \epsilon_k - \alpha_{k-1} (\kappa+\mu) \langle x_{k} - x_{k}^*, x^* - v_{k} \rangle \\
		&\leqslant&  (1-\alpha_{k-1}) (\phi_{k-1} (x^*) +\xi_{k-1} -F^*)+ \epsilon_k + \alpha_{k-1} (\kappa+\mu) \Vert x_{k} - x_{k}^*\Vert \Vert x^* - v_{k} \Vert \\
		&\leqslant&  (1-\alpha_{k-1}) (\phi_{k-1} (x^*) +\xi_{k-1} -F^*)+ \epsilon_k + \alpha_{k-1} \sqrt{2(\kappa+\mu) \epsilon_k} \Vert x^* - v_{k} \Vert\\
		&=&  (1-\alpha_{k-1}) (\phi_{k-1} (x^*) +\xi_{k-1} -F^*)+ \epsilon_k +  \sqrt{2\epsilon_k\gamma_{k} } \Vert x^* - v_{k} \Vert,
\end{eqnarray*}
where the first equality uses the relation~(\ref{vkalt}), the last inequality comes from the strong convexity relation $\varepsilon_k \geq G_k(x_k)-G_k(x_k^\star) \geq (1/2)(\kappa+\mu)\|x_k^\star-x_k\|^2$, and the last equality uses the relation~$\gamma_k = (\kappa+\mu)\alpha_{k-1}^2$.

Dividing both sides by $\lambda_{k}$ yields
$$ \frac{1}{\lambda_{k}} (\phi_{k} (x^*) + \xi_{k} - F^*) \leqslant \frac{1}{\lambda_{k-1}} (\phi_{k-1} (x^*) +\xi_{k-1} -F^*)+ \frac{\epsilon_k}{\lambda_{k}} +  \frac{\sqrt{2 \epsilon_k \gamma_{k} } }{\lambda_{k}} \Vert x^* - v_{k} \Vert. $$
A simple recurrence gives,  
$$ \frac{1}{\lambda_{k}} (\phi_{k} (x^*) + \xi_{k} - F^*) \leqslant \phi_{0} (x^*) -F^*+ \sum_{i=1}^{k}\frac{\epsilon_i}{\lambda_{i}} + \sum_{i=1}^{k} \frac{ \sqrt{2 \epsilon_i \gamma_{i} }} {\lambda_{i}} \Vert x^* - v_{i} \Vert. $$
Finally, by lemmas~\ref{estimate} and~\ref{xik}, 
\begin{displaymath}
  \phi_{k}(x^*) + \xi_{k} -F^*  = \,\, \frac{\gamma_{k}}{2} \Vert x^* - v_{k} \Vert^2 + \phi_{k}^*+ \xi_{k} -F^* \geqslant  \frac{\gamma_{k}}{2} \Vert x^* - v_{k} \Vert^2+ F(x_{k}) -F^*.
 \end{displaymath}
As a result,
\begin{equation} \label{error}
 \frac{1}{\lambda_{k}} (F(x_{k}) - F^* + \frac{\gamma_{k}}{2} \Vert x^* - v_{k} \Vert^2) \leqslant \phi_{0} (x^*) -F^*+ \sum_{i=1}^{k}\frac{\epsilon_i}{\lambda_{i}} + \sum_{i=1}^{k} \frac{ \sqrt{2 \epsilon_i \gamma_{i} }} {\lambda_{i}} \Vert x^* - v_{i} \Vert.
 \end{equation}
\end{proof}

To control the error term on the right and finish the proof of Theorem~\ref{thm}, we are going to borrow some methodology used to analyze the convergence of
inexact proximal gradient algorithms from~\cite{proxinexact}, and use an extension of a lemma presented in~\cite{proxinexact} to bound the value of $\Vert v_i - x^* \Vert $. 
This lemma is presented below.
\begin{lemma}[\bfseries Simple Lemma on Non-Negative Sequences]\label{lemma1.3}~\newline
   Assume that the nonnegative sequences $(u_k)_{k \geq 0}$ and $(a_k)_{k \geq 0}$ satisfy the following recursion for all $k \geq 0$:
\begin{equation}\label{star}
 u_k^2 \leqslant S_k + \sum_{i=1}^{k} a_i u_i,
\end{equation}
where $(S_k)_{k \geq 0}$ is an increasing sequence such that $S_0 \geqslant u_0^2$. Then,
\begin{equation}
   u_k \leqslant \frac{1}{2} \sum_{i=1}^{k} a_i  + \sqrt{ \Big (\frac{1}{2} \sum_{i=1}^k a_i \Big )^2+ S_k}.\label{eq:uk}
\end{equation}
Moreover, 
$$ S_k + \sum_{i=1}^{k} a_i u_i \leqslant \left( \sqrt{S_k}+ \sum_{i=1}^{k} a_i \right)^2.$$
\end{lemma}

\begin{proof}
   The first part---that is, Eq.~(\ref{eq:uk})---is exactly Lemma~1
   from~\cite{proxinexact}. The proof is in their appendix.
   Then, by calling~$b_k$ the right-hand side of~(\ref{eq:uk}), we have that for all $k \geqslant 1$, $u_k \leqslant b_k$. Furthermore $(b_k)_{k \geq 0}$ is increasing and we have 
\begin{eqnarray*}
S_k + \sum_{i=1}^{k} a_i u_i \leqslant S_k + \sum_{i=1}^{k} a_i b_i \leqslant S_k + \Big (\sum_{i=1}^{k} a_i \Big) b_k = b_k^2,
\end{eqnarray*}
and using the inequality $\sqrt{x+y} \leqslant \sqrt{x} + \sqrt{y}$, we have
\begin{equation*}
 b_k = \frac{1}{2} \sum_{i=1}^{k} a_i  + \sqrt{ \Big (\frac{1}{2} \sum_{i=1}^k a_i \Big )^2+ S_k} 
 	\,\,\leqslant \,\,\frac{1}{2} \sum_{i=1}^{k} a_i  + \sqrt{ \Big (\frac{1}{2} \sum_{i=1}^k a_i \Big )^2}+ \sqrt{S_k} 
	\,\,=\,\,\sqrt{S_k}+ \sum_{i=1}^{k} a_i.
\end{equation*}
As a result, 
$$ S_k + \sum_{i=1}^{k} a_i u_i  \leqslant b_k^2 \leqslant  \left(\sqrt{S_k}+ \sum_{i=1}^{k} a_i \right)^2. $$
\end{proof}

We are now in shape to conclude the proof of Theorem~\ref{thm}. We apply the previous lemma to (\ref{error}):
$$ \frac{1}{\lambda_{k}} \Big (\frac{\gamma_{k}}{2} \Vert x^* - v_{k} \Vert^2+ F(x_{k}) -F^* \Big ) \leqslant  \phi_{0} (x^*)  -F^*+ \sum_{i=1}^{k}\frac{\epsilon_i}{\lambda_{i}} + \sum_{i=1}^{k} \frac{ \sqrt{2 \epsilon_i \gamma_{i} }} {\lambda_{i}}  \Vert x^* - v_{i} \Vert.$$
Since $F(x_{k}) -F^* \geqslant 0$, we have 
$$ \underbrace {\frac{\gamma_{k}}{2 \lambda_{k}} \Vert x^* - v_{k} \Vert^2 }_{u_k^2} \leqslant \underbrace{\phi_{0} (x^*) -F^*+ \sum_{i=1}^{k}\frac{\epsilon_i}{\lambda_{i}}}_{S_k} + \sum_{i=1}^{k} \underbrace{\frac{ \sqrt{2 \epsilon_i \gamma_{i} }} {\lambda_{i}} \Vert x^* - v_{i} \Vert}_{a_i u_i},$$
with
$$   u_i  = \sqrt{\frac{\gamma_{i}}{2\lambda_{i}}} \Vert x^* - v_{i} \Vert  \quad  \text{ and }\quad a_i = 2 \sqrt {\frac{\epsilon_i }{\lambda_{i}}} ~~\text{and}~~
 S_k = \phi_{0} (x^*)  -F^*+ \sum_{i=1}^{k}\frac{\epsilon_i}{\lambda_{i}}.$$
Then by Lemma~{\ref{lemma1.3}}, we have
\begin{equation*}
 F(x_{k}) -F^* \leqslant  \,\,\lambda_{k} \left ( S_k +  \sum_{i=1}^k a_i u_i \right ) 
\leqslant \,\, \lambda_{k} \left( \sqrt{S_k}+ \sum_{i=1}^{k} a_i \right)^2 
= \,\, \lambda_{k} \left( \sqrt{S_k}+2 \sum_{i=1}^{k} \sqrt {\frac{\epsilon_i }{\lambda_{i}}} \right)^2,
\end{equation*}
which is the desired result.

\section{Proofs of the Main Theorems and Propositions}\label{sec:proofs}
\subsection{Proof of Theorem \ref{convergence}}
\begin{proof}
   We simply use Theorem~\ref{thm} and specialize it to the choice of parameters.
   The initialization $\alpha_0=\sqrt{q}$ leads to a particularly simple form of the
   algorithm, where~$\alpha_k=\sqrt{q}$ for all~$k \geq 0$.
   Therefore, the sequence~$(\lambda_k)_{k \geq 0}$ from Theorem~\ref{thm} is also
   simple. For all~$k \geq 0$, we indeed have $\lambda_k  = (1-\sqrt{q})^{k}$. 
   To upper-bound the quantity~$S_k$ from Theorem~\ref{thm}, we now remark that $\gamma_0=\mu$ and thus, by strong convexity of~$F$,
$$ F(x_0) + \frac{\gamma_0}{2} \Vert x_0 - x^* \Vert^2 - F^* \leqslant 2 (F(x_0) -F^*).$$
Therefore,
\begin{equation*}
   \begin{split}
      \sqrt{S_k}+2 \sum_{i=1}^{k} \sqrt {\frac{\epsilon_i }{\lambda_{i}}}  & =  \sqrt{F(x_0) + \frac{\gamma_0}{2} \Vert x_0 - x^* \Vert^2 - F^*+ \sum_{i=1}^{k}\frac{\epsilon_i}{\lambda_{i}}} +2 \sum_{i=1}^{k} \sqrt {\frac{\epsilon_i }{\lambda_{i}}} \\
 &\leqslant  \sqrt{ F(x_0) + \frac{\gamma_0}{2} \Vert x_0 - x^* \Vert^2  -F^*}+ 3 \sum_{i=1}^{k} \sqrt {\frac{\epsilon_i }{\lambda_{i}}} \\
  & \leqslant  \sqrt{ 2(F(x_0)   -F^*)}+ 3 \sum_{i=1}^{k} \sqrt {\frac{\epsilon_i }{\lambda_{i}}} \\
 & =  \sqrt{ 2(F(x_0)   -F^*)} \left[ 1+  \sum_{i=1}^{k}  {\underbrace{\left (\sqrt {\frac{1- \rho }{1- \sqrt{q}}}\right ) }_{\eta}}^{i}\right] \\
 & =  \sqrt{ 2(F(x_0)   -F^*)} \,\, \frac{\eta^{k+1}-1}{\eta-1} \\
 & \leqslant  \sqrt{ 2(F(x_0)   -F^*)} \,\, \frac{\eta^{k+1}}{\eta-1}.
   \end{split}
\end{equation*}

Therefore, Theorem~\ref{thm} combined with the previous inequality gives us
\begin{equation*}
   \begin{split}
      F(x_{k}) -F^*  & \leqslant  2\lambda_{k} (F(x_0)   -F^*) \,\, \left (\frac{\eta^{k+1}}{\eta-1} \right )^2 \\
                     & =  2\left ( \frac{\eta}{\eta-1} \right )^2 (1-\rho)^{k}(F(x_0)  -F^*) \\
                     & =  2\left ( \frac{\sqrt{1-\rho}}{\sqrt{1-\rho} - \sqrt{1- \sqrt{q}}} \right )^2 (1-\rho)^{k}(F(x_0)  -F^*) \\
                     & =  2\left ( \frac{1}{\sqrt{1-\rho} - \sqrt{1- \sqrt{q}}} \right )^2 (1-\rho)^{k+1}(F(x_0)  -F^*).
   \end{split}
\end{equation*}
Since $\sqrt{1-x} + \frac{x}{2} $ is decreasing in $[0,1]$, we have $\sqrt{1-\rho}+\frac{\rho}{2} \geqslant \sqrt{1-\sqrt{q}} + \frac{\sqrt{q}}{2}$. Consequently,
$$ F(x_{k}) -F^* \leqslant \frac{8}{(\sqrt{q} -\rho)^2} (1-\rho)^{k+1} (F(x_0)  -F^*). $$
\end{proof}
\subsection{Proof of Proposition~\ref{prop:complexity}}
To control the number of calls of $\mtd$, we need to upper bound
$G_k(x_{k-1})-G_k^*$ which is given by the following lemma:
\begin{lemma}[\bfseries Relation between~$G_k(x_{k-1})$ and~$\varepsilon_{k-1}$]~\label{6.11}\newline
   Let~$(x_k)_{k \geq 0}$ and~$(y_k)_{k \geq 0}$ be generated by Algorithm~\ref{alg:catalyst}.
   Remember that by definition of~$x_{k-1}$, 
$$ G_{k-1}(x_{k-1}) - G_{k-1}^* \leqslant \epsilon_{k-1}.$$ 
Then, we have 
\begin{equation} \label{bound}
G_k(x_{k-1}) - G_k^* \leqslant 2\epsilon_{k-1}+ \frac{\kappa^2}{\kappa+\mu} \Vert y_{k-1} - y_{k-2} \Vert^2.
\end{equation}
\end{lemma}

\begin{proof}
   We first remark that for any $x,y$ in $\mathbb{R}^p$, we have
  $$  G_k(x) - G_{k-1}(x) = G_k(y) - G_{k-1}(y) + \kappa \langle y-x, y_{k-1} - y_{k-2} \rangle, \quad \forall k \geqslant 2,    $$
  which can be shown by using the respective definitions of $G_k$ and~$G_{k-1}$ and manipulate the quadratic term resulting from the difference $G_k(x)-G_{k-1}(x)$.

Plugging $x = x_{k-1}$ and $y = x_k^*$ in the previous relation yields
\begin{equation}\label{eq:ineq1}
   \begin{split}
G_k(x_{k-1}) -G_k^* &= G_{k-1}(x_{k-1}) -G_{k-1}(x_k^*) + \kappa \langle x_k^* - x_{k-1}, y_{k-1} - y_{k-2} \rangle \\
				&= G_{k-1}(x_{k-1}) -G_{k-1}^* + G_{k-1}^*-G_{k-1}(x_k^*) + \kappa \langle x_k^* - x_{k-1}, y_{k-1} - y_{k-2} \rangle  \\
    &\leqslant \varepsilon_{k-1}   + G_{k-1}^*-G_{k-1}(x_k^*) + \kappa \langle x_k^* - x_{k-1}, y_{k-1} - y_{k-2} \rangle  \\
				&\leqslant \epsilon_{k-1} - \frac{\mu+\kappa}{2} \Vert x_k^* - x_{k-1}^* \Vert^2 + \kappa \langle x_k^* - x_{k-1}, y_{k-1} - y_{k-2} \rangle,
   \end{split}
\end{equation}
where the last inequality comes from the strong convexity inequality of 
$$G_{k-1}(x_k^\star) \geq G_{k-1}^\star + \frac{\mu+\kappa}{2}\|x_k^\star-x_{k-1}^\star \|^2.$$
Moreover, from the inequality ${\langle x,y\rangle} \leq \frac{1}{2}\|x\|^2 + \frac{1}{2}\|y\|^2$, we also have
\begin{equation}
   \kappa \langle x_k^* -x_{k-1}^*, y_{k-1} - y_{k-2} \rangle \leqslant  \frac{\mu+\kappa}{2} \Vert x_k^* - x_{k-1}^* \Vert^2 + \frac{\kappa^2}{2(\kappa+\mu)} \Vert y_{k-1} -y_{k-2} \Vert^2,\label{eq:ineq2}
\end{equation}
and
\begin{equation}\label{eq:ineq3}
   \begin{split}
\kappa \langle  x_{k-1}^*- x_{k-1}, y_{k-1} - y_{k-2} \rangle &\leqslant  \frac{\mu+\kappa}{2} \Vert x_{k-1}^* - x_{k-1} \Vert^2 + \frac{\kappa^2}{2(\kappa+\mu)} \Vert y_{k-1} -y_{k-2} \Vert^2 \\
                                                              & \leqslant \epsilon_{k-1} + \frac{\kappa^2}{2(\kappa+\mu)} \Vert y_{k-1} -y_{k-2} \Vert^2.
   \end{split}
\end{equation}
Summing inequalities~(\ref{eq:ineq1}), (\ref{eq:ineq2}) and~(\ref{eq:ineq3}) gives the desired result.
\end{proof}

Next, we need to upper-bound the term $\Vert y_{k-1} - y_{k-2} \Vert^2$, which was also required in the convergence proof
of the accelerated SDCA algorithm~\cite{accsdca}. We follow here their methodology.
\begin{lemma}[\bfseries Control of the term~$\Vert y_{k-1} - y_{k-2} \Vert^2$.]~\label{6.13}\newline
   Let us consider the iterates~$(x_k)_{k \geq 0}$ and~$(y_k)_{k \geq 0}$ produced by Algorithm~\ref{alg:catalyst}, and define 
$$ \delta_k = C (1-\rho)^{k+1} (F(x_0) -F^*),$$
which appears in Theorem~\ref{convergence} and which is such that $F(x_k)-F^* \leqslant \delta_k$. 
Then, for any $k \geqslant 3$, 
$$\Vert y_{k-1} - y_{k-2} \Vert^2 \leqslant \frac{72}{\mu} \delta_{k-3}.$$
\end{lemma}
\begin{proof}
   We follow here~\cite{accsdca}. By definition of $y_k$, we have
\begin{equation*}
   \begin{split}
\Vert y_{k-1} - y_{k-2} \Vert  &= \Vert x_{k-1} + \beta_{k-1} (x_{k-1} - x_{k-2}) - x_{k-2} - \beta_{k-2} (x_{k-2} - x_{k-3}) \Vert \\
                               &\leqslant (1+\beta_{k-1}) \Vert x_{k-1} - x_{k-2} \Vert + \beta_{k-2} \Vert x_{k-2} - x_{k-3} \Vert \\
				     &\leqslant 3 \max \left \{ \Vert x_{k-1} - x_{k-2} \Vert, \Vert x_{k-2} - x_{k-3} \Vert  \right \},
   \end{split}
\end{equation*}
where~$\beta_k$ is defined in~(\ref{y_k}). 
The last inequality was due to the fact that~$\beta_k \leq 1$.  Indeed, the
specific choice of~$\alpha_0=\sqrt{q}$ in Theorem~\ref{thm} leads to~$\beta_k = \frac{\sqrt{q}-{q}}{\sqrt{q}+q} \leqslant 1$ for
all~$k$. Note, however, that this relation~$\beta_k \leq 1$ is true regardless
of the choice of~$\alpha_0$:
\begin{displaymath}
   \beta_k^2 = \frac{\left(\alpha_{k-1} - \alpha_{k-1}^2\right)^2}{\left(\alpha_{k-1}^2 + \alpha_k\right)^2} = 
\frac{\alpha_{k-1}^2 + \alpha_{k-1}^4 - 2\alpha_{k-1}^3}{\alpha_{k}^2  + 2\alpha_k \alpha_{k-1}^2 + \alpha_{k-1}^4} = 
\frac{\alpha_{k-1}^2 + \alpha_{k-1}^4 - 2\alpha_{k-1}^3}{\alpha_{k-1}^2 + \alpha_{k-1}^4 + q\alpha_k + \alpha_k \alpha_{k-1}^2} \leq 1,
\end{displaymath}
where the last equality uses the relation $\alpha_{k}^2 + \alpha_{k} \alpha_{k-1}^2= \alpha_{k-1}^2 + q \alpha_{k}$ from Algorithm~\ref{alg:catalyst}.
To conclude the lemma, we notice that by triangle inequality
$$ \Vert x_{k} - x_{k-1} \Vert \leqslant  \Vert x_{k}-x^* \Vert + \Vert x_{k-1}-x^* \Vert,$$
and by strong convexity of~$F$
$$  \frac{\mu}{2}\Vert x_k - x^* \Vert^2 \leqslant F(x_k) -F(x^*)  \leqslant \delta_{k}.$$
As a result, 
\begin{align*}
\Vert y_{k-1} - y_{k-2} \Vert^2 & \leqslant 9 \max \left \{ \Vert x_{k-1} - x_{k-2} \Vert^2, \Vert x_{k-2} - x_{k-3} \Vert^2  \right \} \\
					& \leqslant 36 \max \left \{ \Vert x_{k-1} - x^* \Vert^2, \Vert x_{k-2} - x^* \Vert^2, \Vert x_{k-3} - x^* \Vert^2   \right \} \\
					& \leqslant \frac{72}{\mu} \delta_{k-3}.
\end{align*}
\end{proof}

We are now in shape to conclude the proof of Proposition~\ref{prop:complexity}.

By Proposition~\ref{6.11} and lemma~\ref{6.13}, we have for all $k \geqslant 3$,
$$   G_k(x_{k-1}) - G_k^* \leqslant 2\epsilon_{k-1}+ \frac{\kappa^2}{\kappa+\mu} \frac{72}{\mu} \delta_{k-3} \leqslant  2\epsilon_{k-1}+  \frac{72 \kappa}{\mu} \delta_{k-3}.$$
Let $(z_t)_{t \geq 0}$ be the sequence of using $\mtd$ to solve $G_k$ with initialization $z_0=x_{k-1}$. 
By assumption~(\ref{eq:rategk}), we have
$$ G_k(z_t) - G_k^* \leqslant A \, (1-\tau_{\mtd})^t (G_k(x_{k-1}) - G_k^*) \leqslant A \, e^{-\tau_{\mtd}t} (G_k(x_{k-1}) - G_k^*).$$
The number of iterations $T_\mtd$ of $\mtd$ to guarantee an accuracy of $\epsilon_k$ needs to satisfy
$$   A \, e^{-\tau_{\mtd}T_{\mtd}} (G_k(x_{k-1}) - G_k^*) \leqslant \epsilon_k,  $$
which gives
\begin{equation}
\label{eq:def-T-mtd-direct}
 T_{\mtd} = \left \lceil \frac{1}{\tau_{\mtd}} \ln \left ( \frac{ A (G_k(x_{k-1}) - G_k^*)}{\epsilon_k}  \right ) \right \rceil.
 \end{equation}
Then, it remains to upper-bound
\begin{equation*}
  \frac{G_k(x_{k-1}) -G_k^*}{\epsilon_k}  \leqslant \frac{2\epsilon_{k-1}+ \frac{72 \kappa}{\mu} \delta_{k-3} }{\epsilon_k} 
 									 = \frac{2 (1-\rho)+  \frac{72 \kappa}{\mu} \cdot \frac{9C}{2} }{(1-\rho)^2} 
 									 = \frac{2}{1-\rho}+ \frac{2592 \kappa}{\mu (1-\rho)^2 (\sqrt{q} - \rho)^2}.
\end{equation*}
Let us denote $R$ the right-hand side. We remark that this upper bound holds for $k \geqslant 3$. We now consider the cases $k=1$ and $k=2$.

When $k=1$, $G_1(x) = F(x) + \frac{\kappa}{2} \Vert x- y_0 \Vert^2$. Note that $x_0=y_0$, then $G_1(x_0) = F(x_0)$. As a result, 
$$G_1(x_0) -G_1^*= F(x_0)-F(x_1^*)-\frac{\kappa}{2} \Vert x_1^* - y_0 \Vert^2 \leqslant F(x_0)-F(x_1^*) \leqslant F(x_0)-F^*.$$
Therefore,
$$ \frac{G_1(x_0) -G_1^*}{\epsilon_1} \leqslant \frac{F(x_0)-F^*}{\epsilon_1} = \frac{9}{2(1-\rho)} \leqslant R.$$

When $k=2$, we remark that $ y_1 - y_0 = (1+\beta_1) (x_1-x_0)$. Then, by following similar steps as in the proof of Lemma~\ref{6.13}, we have
$$ \Vert y_1 - y_0 \Vert^2 \leqslant 4 \Vert x_1 -x_0 \Vert^2 \leqslant \frac{32 \delta_0}{\mu},$$
which is smaller than $ \frac{72 \delta_{-1}}{\mu}$. Therefore, the previous steps from the case~$k \geq 3$ apply and 
$\frac{G_2(x_1) -G_2^*}{\epsilon_2} \leqslant R$.
Thus, for any $k \geqslant 1$,  
\begin{equation}
\label{eq:def-T-mtd-major}
   T_{\mtd} \leqslant  \left \lceil \frac{ \ln \left ( A R \right ) }{\tau_{\mtd}} \right \rceil,
\end{equation}
which concludes the proof.

\subsection{Proof of Theorem~\ref{convergence2}.}
We will again Theorem~\ref{thm} and specialize it to the choice of parameters.
To apply it, the following Lemma will be useful to control the growth
of~$(\lambda_k)_{k \geq 0}$.
\begin{lemma}[\bfseries Growth of the Sequence~$(\lambda_k)_{k \geq 0}$]~\label{lemma:lambdak}\newline
   Let $(\lambda_k)_{k \geq 0}$ be the sequence defined in~(\ref{eq:lambdak})
   where~$(\alpha_k)_{k \geq 0}$ is produced by Algorithm~\ref{alg:catalyst}
   with~$\alpha_0=\frac{\sqrt{5}-1}{2}$ and~$\mu=0$.
   Then, we have the following bounds for all~$k \geq 0$,
   \begin{displaymath}
      \frac{4}{(k+2)^2} \geq \lambda_k \geq \frac{2}{(k+2)^2}.
   \end{displaymath}
\end{lemma}
\begin{proof}
   Note that by definition of~$\alpha_k$, we have for all~$k \geq 1$,
   \begin{displaymath}
      \alpha_k^2 = (1-\alpha_k) \alpha_{k-1}^2 = \prod_{i=1}^k (1-\alpha_i)\alpha_0^2 = \lambda_{k+1} \frac{\alpha_0^2}{1-\alpha_0} = \lambda_{k+1}.
   \end{displaymath}
   With the choice of~$\alpha_0$, the quantity~$\gamma_0$ defined in~(\ref{eq:gamma0}) is equal to~$\kappa$. By Lemma~\ref{lemma:lambda}, we have $\lambda_k \leq \frac{4}{(k+2)^2}$ for all~$k \geq 0$ and thus $\alpha_k \leq \frac{2}{k+3}$ for all~$k \geq 1$ (it is also easy to check numerically that this is also true for $k=0$ since $\frac{\sqrt{5}-1}{2} \approx 0.62 \leqslant \frac{2}{3}$). We now have all we need to conclude the lemma:
   \begin{displaymath}
      \lambda_k = \prod_{i=0}^{k-1} (1-\alpha_i) \geq \prod_{i=0}^{k-1} \left(1- \frac{2}{i+3}\right) = \frac{2}{(k+2)(k+1)} \geq \frac{2}{(k+2)^2}.
   \end{displaymath}
\end{proof}

With this lemma in hand, we may now proceed and apply Theorem~\ref{thm}.
We have remarked in the proof of the previous lemma that $\gamma_0 = \kappa$. 
Then,
\begin{align*}
   \sqrt{S_k}+2 \sum_{i=1}^{k} \sqrt {\frac{\epsilon_i }{\lambda_{i}}}  \,=\, & \sqrt{F(x_0) - F^\star  + \frac{\kappa}{2}\|x_0-x^\star\|^2   + \sum_{i=1}^{k}\frac{\epsilon_i}{\lambda_{i}}} +2 \sum_{i=1}^{k} \sqrt {\frac{\epsilon_i }{\lambda_{i}}} \\
 \,\leqslant\, & \sqrt{ F(x_0) - F^\star  + \frac{\kappa}{2}\|x_0-x^\star\|^2}+ 3 \sum_{i=1}^{k} \sqrt {\frac{\epsilon_i }{\lambda_{i}}} \\
\,\leqslant  \,& \sqrt{\frac{\kappa}{2} \Vert x_0 - x^*\Vert^2}+ \sqrt{F(x_0)-F^*}\left (1+ \sum_{i=1}^k \frac{1}{(i+2)^{1+\eta/2}} \right),
\end{align*}
where the last inequality uses Lemma~\ref{lemma:lambdak} to upper-bound the ratio $\varepsilon_i/\lambda_i$.
Moreover,
$$  \sum_{i=1}^{k} \frac{1}{(i+2)^{1+\eta/2}}  \leqslant \sum_{i=2}^{\infty} \frac{1}{i ^{1+\eta/2}} \leqslant \int_1^{\infty} \frac{1}{x^{1+\eta/2}} \, \text{d}x = \frac{2}{\eta}.$$
Therefore, by (\ref{F}) from Theorem~\ref{thm},
\begin{align*}
F(x_{k}) -F^* \leqslant \,\,  & \lambda_{k} \left ( \sqrt{S_k}+ 2\sum_{i=1}^{k} \sqrt {\frac{\epsilon_i }{\lambda_{i}}} \right )^2 \\
		    \leqslant \,\, & \frac{4}{(k+2)^2} \left ( \sqrt{F(x_0)-F^*}\left(1+\frac{2}{\eta}\right)+ \sqrt{\frac{\kappa}{2} \Vert x_0 - x^*\Vert^2} \right )^2 \\
		    \leqslant \,\, & \frac{8}{(k+2)^2} \left ( \left( 1+ \frac{2}{\eta}\right)^2 (F(x_0) -F^*)+\frac{\kappa}{2} \Vert x_0 - x^*\Vert^2 \right ).  
\end{align*}
The last inequality uses $(a+b)^2 \leqslant 2(a^2+b^2)$.

\subsection{Proof of Proposition~\ref{prop:complexity2}}

When $\mu=0$, we remark that Proposition~\ref{6.11} still holds but Lemma~\ref{6.13} does not.
The main difficulty is thus to find another way to control the quantity $\Vert y_{k-1} - y_{k-2} \Vert$. 

Since $F(x_k)-F^*$ is bounded by Theorem~\ref{convergence2}, we may use the bounded level set assumptions to ensure that
there exists $B>0$ such that $\Vert x_k - x^* \Vert \leqslant B$ for any $k \geq 0$ where~$x^\star$ is a minimizer of~$F$. 
We can now follow similar steps as in the proof of Lemma~\ref{6.13}, and show that
$$  \Vert y_{k-1} - y_{k-2} \Vert^2 \leqslant 36 B^2.$$
Then by Proposition~\ref{6.11}, 
$$  G_k(x_{k-1})-G_k^* \leqslant 2\epsilon_{k-1} + 36\kappa B^2. $$
Since $\kappa >0$, $G_k$ is strongly convex, then using the same argument as in the strongly convex case, the number of calls for $\mtd$ is given by 
\begin{equation}
   \left \lceil \frac{1}{\tau_{\mtd}} \ln \left ( \frac{ A (G_k(x_{k-1}) - G_k^*)}{\epsilon_k}  \right ) \right \rceil. \label{eq:ratemuzero}
\end{equation}
Again, we need to upper bound it
\begin{equation*}
   \frac{G_k(x_{k-1}) -G_k^* }{\epsilon_k} \leqslant \frac{2\epsilon_{k-1}+ 36\kappa B^2 }{\epsilon_k} = \frac{2 (k+1)^{4+\eta}}{(k+2)^{4+\eta}} + \frac{162 \kappa B^2 (k+2)^{4+\eta}}{(F(x_0)-F^*)}.
\end{equation*}
The right hand side is upper-bounded by $O((k+2)^{4+\eta})$. Plugging this relation into~(\ref{eq:ratemuzero}) gives the desired result.

\section{Derivation of Global Convergence Rates}\label{sec:accel-factor-calculs}
We give here a generic ``template'' for computing the \emph{optimal choice of $\kappa$} to accelerate a given algorithm $\mtd$,
and therefore compute the rate of convergence of the accelerated algorithm $\mathcal{A}$. 

We assume
here that $\mtd$ is a \emph{randomized} first-order optimization algorithm, \textit{i.e.} the iterates $(x_k)$ generated by $\mtd$ 
are a sequence of random variables; specialization to a deterministic algorithm is straightforward. Also, for the sake of simplicity, 
we shall use simple notations to denote the stopping time to reach accuracy $\varepsilon$. Definition and notation using filtrations, 
$\sigma$-algebras, etc. are unnecessary for our purpose here where the quantity of interest has a clear interpretation. 

Assume that algorithm $\mtd$ enjoys a linear rate of convergence, in expectation. There exists constants $\C_{\mtd,F}$ and $\tau_{\mtd,F}$ such that 
the sequence of iterates $(x_k)_{k \geq 0}$ for minimizing a strongly-convex objective~$F$ satisfies 
\begin{equation}
\label{eq:rategk-expect}
\E 
\left[ 
F(x_k) -F^* 
\right] 
\leq
\C_{\mtd,F}
\left(
1 - \tau_{\mtd,F}
\right)^k.
\end{equation}
Define the random variable $T_{\mtd,F}(\varepsilon)$ (stopping time) corresponding 
to the minimum number of iterations to guarantee an accuracy $\varepsilon$ in the course of running $\mtd$ 
\begin{equation}
T_{\mtd,F}(\varepsilon) 
:= \inf 
\left\{
k \geq 1, \: F(x_k) -F^*  \leq \varepsilon
\right\}
\end{equation} 
Then, an upper bound on the expectation is provided by the following lemma.
\begin{lemma}[\bfseries Upper Bound on the expectation of~$T_{\mtd,F}(\varepsilon)$]~\label{lemma:expectation}\newline
   Let $\mtd$ be an optimization method with the expected rate of convergence~(\ref{eq:rategk-expect}). 
Then,
\begin{equation} 
 \label{eq:def-T-mtd-rand-major}
 \E[T_{\mtd}(\varepsilon)] \leqslant \frac{1}{\tau_{\mtd}} \ln \left (\frac{ 2 C_{\mtd}}{  \tau_\mtd  \cdot \varepsilon }\right ) +1 = \Otilde\left( \frac{1}{\tau_{\mtd}} \ln \left (\frac{ C_{\mtd}}{  \varepsilon }\right )\right),
\end{equation}
where we have dropped the dependency in~$F$ to simplify the notation.
\end{lemma}

\begin{proof}
We abbreviate $\tau_\mtd$ by $\tau$. Set 
$$T_0 =  \frac{1}{\tau}\log \left ( \frac{1}{1- e^{-\tau}} \frac{C_{\mtd}}{\varepsilon} \right ). $$ 
For any $k \geq 0$, we have 
$$ \E[F(x_k) - F^*]  \leqslant C_{\mtd} (1-\tau)^k  \leqslant C_{\mtd}  \, e^{-k \tau}. $$
By Markov's inequality,
\begin{equation}
   \p [F(x_k) - F^* > \epsilon] = \p [T_{\mtd}(\epsilon) > k] \leqslant \frac{\E[F(x_k) - F^*]}{\epsilon} \leqslant \frac{C_{\mtd}  \, e^{-k \tau}}{\epsilon}.
\end{equation}
Together with the fact $\p \leqslant 1$ and $k \geq 0$. We have
$$ \p [T_{\mtd}(\epsilon) \geqslant k+1] \leqslant \min \left \{ \frac{C_{\mtd}}{\epsilon} e^{- k \tau } ,1 \right \}. $$
Therefore,
\begin{align*}
   \E[T_{\mtd}(\epsilon)] & = \sum_{k=1}^{\infty} \p[T_{\mtd}(\epsilon) \geqslant k] = \sum_{k=1}^{T_0} \p[T_{\mtd}(\epsilon) \geqslant k] + \sum_{k=T_0+1}^{\infty} \p[T_{\mtd}(\epsilon) \geqslant k] \\
                          & \leqslant T_0 + \sum_{k=T_0}^{\infty} \frac{C_{\mtd}}{\epsilon} e^{-k \tau } = T_0  +  \frac{C_{\mtd}}{\epsilon}  e^{- T_0 \tau} \sum_{k=0}^{\infty} e^{-k \tau} \\
                          &  = T_0  +  \frac{C_{\mtd}}{\epsilon}  \frac{ e^{- \tau T_0}}{1- e^{-\tau}}  = T_0 +1.
\end{align*}
As simple calculation shows that for any $\tau \in (0,1)$, $\frac{\tau}{2} \leqslant 1-e^{-\tau}$ and then
$$ \E[T_{\mtd}(\epsilon)]  \leqslant T_0 +1 =  \frac{1}{\tau}\log \left ( \frac{1}{1- e^{-\tau}} \frac{C_{\mtd}}{\epsilon} \right )+1 \leqslant  \frac{1}{\tau}\log \left ( \frac{2 C_{\mtd}}{ \tau \epsilon} \right )+1.$$
\end{proof}

Note that the previous lemma mirrors Eq.~(\ref{eq:def-T-mtd-direct}-\ref{eq:def-T-mtd-major}) in the proof 
of Prop.~\ref{convergence} in Appendix~\ref{sec:proofs}.  For all optimization
methods of interest, the rate~$\tau_{\mtd,G_k}$ is independent of~$k$ and
varies with the parameter~$\kappa$.
We may now compute the iteration-complexity (in expectation) of the accelerated
algorithm~$\Acal$---that is, for a given~$\varepsilon$, the expected total
number of iterations performed by the method~$\mtd$.  Let us now
fix~$\varepsilon > 0$.
Calculating the iteration-complexity decomposes into three steps:
\begin{enumerate}
\item Find $\kappa$ that maximizes the ratio $\tau_{\mtd,G_k}/\sqrt{\mu
   +\kappa}$~for algorithm $\mtd$ when~$F$ is~$\mu$-strongly convex. In the
   non-strongly convex case, we suggest maximizing instead the ratio
   $\tau_{\mtd,G_k}/\sqrt{L+\kappa}$. Note that the choice of~$\kappa$ is less
   critical for non-strongly convex problems since it only affects
   multiplicative constants in the global convergence rate.
\item Compute the upper-bound of the number of outer iterations
   $k_{\text{out}}$ using Theorem~\ref{convergence} (for the strongly convex
   case), or Theorem~\ref{convergence2} (for the non-strongly convex case), by
   replacing $\kappa$ by the optimal value found in step 1.
\item Compute the upper-bound of the expected number of inner iterations 
   $$\max_{k=1,\ldots,k_{\text{out}}}\E [T_{\mtd,G_k}(\varepsilon_k)] \leq k_{\text{in}},$$
by replacing the appropriate quantities in Eq.~\ref{eq:def-T-mtd-rand-major}
for algorithm $\mtd$; for that purpose, the proofs of
Propositions~\ref{prop:complexity} of~\ref{prop:complexity2} my be used to
upper-bound the ratio~$\C_{\mtd,G_k}/\varepsilon_k$, or another dedicated analysis for~$\mtd$ may be required if the constant~$\C_{\mtd,G_k}$ does not have the required form~$A (G_k(z_{0})-G_k^\star)$ in~(\ref{eq:rategk}).
\end{enumerate}
Then, the iteration-complexity (in expectation) denoted $\text{Comp.}$ is given by
\begin{equation}
\text{Comp} \leq k_{\text{in}} \times k_{\text{out}} \; .
\end{equation}

\section{A Proximal MISO/Finito Algorithm}\label{appendix:miso}
In this section, we present the algorithm MISO/Finito, and show how to extend
it in two ways.  First, we propose a proximal version to deal with composite
optimization problems, and we analyze its rate of convergence.
Second, we show how to remove a large sample condition~$n \geq 2L/\mu$, which was
necessary for the convergence of the algorithm. The resulting algorithm is a
variant of proximal SDCA~\cite{sdca} with a different stepsize and a stopping
criterion that does not use duality. 

\subsection{The Original Algorithm MISO/Finito}
MISO/Finito was proposed in \cite{miso} and~\cite{finito} for solving the following smooth unconstrained convex minimization problem 
\begin{equation}
   \min_{x \in \Real^p} \left\{ f(x) \defin \frac{1}{n} \sum_{i=1}^{n} f_i(x) \right \}, \label{eq:objmiso}
\end{equation}
where each $f_i$ is differentiable with $L$-Lipschitz continuous derivatives and $\mu$-strongly convex. 
At iteration~$k$, the algorithm updates a list of lower bounds $d_i^{k}$ of the
functions~$f_i$, by randomly picking up one index~$i_k$ among $\{ 1, \cdots, n\}$ and performing the following update
\begin{displaymath}
   d_i^k(x) = \left\{ 
      \begin{array}{ll}
         f_i(x_{k-1}) + \langle \nabla f_i(x_{k-1}) ,x - x_{k-1} \rangle + \frac{\mu}{2}\|x-x_{k-1}\|^2 & \text{if}~~~  i=i_k \\
         d_i^{k-1}(x) & \text{otherwise} 
      \end{array}\right.,
\end{displaymath}
which is a lower bound of~$f_i$ because of the $\mu$-strong convexity of~$f_i$.
Equivalently, one may perform the following updates
\begin{displaymath}
   z_i^k = \left\{ 
      \begin{array}{ll}
         x_{k-1} - \frac{1}{\mu} \nabla f_i(x_{k-1}) & \text{if}~~~  i=i_k \\
         z_i^{k-1} & \text{otherwise} 
      \end{array}\right.,
\end{displaymath}
and all functions $d_i^k$ have the form
\begin{displaymath}
   d_i^k(x) =   c_i^k + \frac{\mu}{2}\|x-z_{i}^k\|^2,
\end{displaymath}
where~$c_i^k$ is a constant.
Then, MISO/Finito performs the following minimization to produce the iterate~$(x_k)$:
\begin{displaymath}
   x_k = \argmin_{x \in \Real^p} \frac{1}{n}\sum_{i=1}^n d_i^k(x) = \frac{1}{n}\sum_{i=1}^n z_i^k,
\end{displaymath}
which is equivalent to
\begin{displaymath}
   x_k \leftarrow x_{k-1} - \frac{1}{n}\left( z_{i_k}^k - z_{i_k}^{k-1}\right).
\end{displaymath}
In many machine learning problems, it is worth remarking that each function~$f_i(x)$ has the specific form $f_i(x) =
l_i(\langle x, w_i \rangle) + \frac{\mu}{2}\|x\|^2$. In such cases, the vectors~$z_i^k$ can
be obtained by storing only~$O(n)$ scalars.\footnote{Note that even though we call this algorithm MISO (or Finito), it was called
MISO$\mu$ in~\cite{miso}, whereas ``MISO'' was originally referring to an
incremental majorization-minimization procedure that uses upper bounds of the
functions~$f_i$ instead of lower bounds, which is appropriate for non-convex optimization problems.}
The main convergence result of~\cite{miso} is that the procedure above 
converges with a linear rate of convergence of the form~(\ref{mtd}),
with~$\tau_{\textrm{MISO}} = 1/3n$ (also refined in $1/2n$ in~\cite{finito}), when
the large sample size constraint~$n\geq 2L/\mu$ is satisfied.

Removing this condition and extending MISO to the composite optimization problem~(\ref{eq:obj}) 
is the purpose of the next section.

\subsection{Proximal MISO}
We now consider the composite optimization problem below,
$$   \min_{x \in \Real^p} \left \{ F(x) = \frac{1}{n} \sum_{i=1}^{n} f_i(x) +\reg(x) \right \},$$
where the functions~$f_i$ are differentiable with~$L$-Lipschitz derivatives and~$\mu$-strongly convex.
As in typical composite optimization problems, $\reg$ is convex but not
necessarily differentiable. We assume that the proximal operator of $\psi$ can
be computed easily. 
The algorithm needs to be initialized with some lower bounds for the functions~$f_i$:
\begin{equation} \label{A1}
    f_i(x) \geqslant \frac{\mu}{2} \Vert x - z_i^0 \Vert^2 + c_i^0, \tag{A1}
\end{equation}
which are guaranteed to exist due to the~$\mu$-strong convexity of~$f_i$. 
For typical machine learning applications, such initialization is easy.
For example, logistic regression with $\ell_2$-regularization satisfies
(\ref{A1}) with $z_i^0=0$ and $c_i^0 = 0$. Then, the MISO-Prox scheme is given in Algorithm~\ref{alg:miso}.  Note that if no simple initialization is
available, we may consider any initial estimate~$\barz_0$ in~$\Real^p$ and
define $z_i^0 = \barz_0 - ({1}/{\mu}) \nabla f_i (\barz_0)$, which requires
performing one pass over the data.

Then, we remark that under the large sample size condition $n \geq 2L/\mu$, we
have~$\delta=1$ and the update of the quantities~$z_i^k$
in~(\ref{eq:misoupdate}) is the same as in the original MISO/Finito algorithm.
As we will see in the convergence analysis, the choice of~$\delta$ ensures convergence
of the algorithm even in the small sample size regime~$n < 2L/\mu$.

 \begin{algorithm}
    \caption{MISO-Prox: an improved MISO algorithm with proximal support.}\label{alg:miso}
     \begin{algorithmic}[1]
      \INPUT $(z_i^0)_{i=1,\ldots,n}$ such that~(\ref{A1}) holds;  $N$ (number of iterations);
        \STATE initialize $\barz_0=\frac{1}{n}\sum_{i=1}^n z_i^0$ and $x_0 = \prox_{\psi/\mu}[\barz_0]$;
        \STATE define $\delta = \min\left(1,\frac{\mu n}{2(L-\mu)}\right)$;
        \FOR{ $k=1,\ldots,N$}
        \STATE randomly pick up an index~$i_k$ in~$\{1,\ldots,n\}$;
        \STATE update
        \begin{equation}\label{eq:misoupdate}
           \begin{split}
              z_i^k & = \left\{ 
              \begin{array}{ll}
                 (1-\delta)z_i^{k-1} + \delta\left(x_{k-1} - \frac{1}{\mu} \nabla f_i(x_{k-1})\right) & \text{if}~~~  i=i_k \\
                 z_i^{k-1} & \text{otherwise} 
              \end{array}\right. \\
              \barz_k & = \barz_{k-1} - \frac{1}{n}\left( z_{i_k}^k - z_{i_k}^{k-1}\right) = \frac{1}{n}\sum_{i=1}^n z_{i}^k\\
              x_k & = \prox_{\psi/\mu}[\barz_k].
           \end{split}
        \end{equation}
        \ENDFOR
        \OUTPUT $x_N$ (final estimate).
  \end{algorithmic}
 \end{algorithm}

 \paragraph{Relation with Proximal SDCA~\cite{sdca}.}
The algorithm MISO-Prox is almost identical to variant~$5$ of proximal
SDCA~\cite{sdca}, which performs the same updates with~$\delta = \mu n/(L+\mu n)$ instead
of~$\delta=\min(1,\frac{\mu n}{2(L-\mu)})$. It is however not clear that 
MISO-Prox actually performs dual ascent steps in the sense of SDCA since the proof
of convergence of SDCA cannot be directly modified to use the stepsize of
proximal MISO and furthermore, the convergence proof of MISO-Prox does not use
the concept of duality.
Another difference lies in the optimality certificate of the
algorithms. Whereas Proximal-SDCA provides a certificate in terms of linear convergence
of a duality gap based on Fenchel duality, Proximal-SDCA ensures linear
convergence of a gap that relies on strong convexity but not on the Fenchel dual (at least explicitly).

\paragraph{Optimality Certificate and Stopping Criterion.}
Similar to the original MISO algorithm, Proximal MISO maintains a list~$(d_i^k)$ of lower bounds of the functions~$f_i$, 
which are updated in the following fashion
\begin{equation}\label{eq:dk}
   d_i^k(x) = \!\left\{ 
      \begin{array}{ll}
         (1-\delta)d_i^{k-1}(x) \!+ \delta\left(f_i(x_{k-1}) \!+\! \langle \nabla f_i(x_{k-1}) ,x - x_{k-1} \rangle \!+\! \frac{\mu}{2}\|x-x_{k-1}\|^2\right) & \text{if}~~~  i=i_k \\
         d_i^{k-1}(x) & \text{otherwise} 
      \end{array}\right.
\end{equation}
Then, the following function is a lower bound of the objective~$F$:
\begin{equation}
   D_k(x) = \frac{1}{n}\sum_{i=1}^n d_i^k(x)  + \psi(x), \label{eq:Dk}
\end{equation}
and the update~(\ref{eq:misoupdate}) can be shown to exactly minimize~$D_k$. As a lower bound of~$F$, we have that~$D_k(x_k) \leq F^\star$ and thus
\begin{displaymath}
   F(x_k) - F^\star \leq F(x_k) - D_k(x_k).
\end{displaymath}
The quantity $F(x_k) - D_k(x_k)$ can then be interpreted as an optimality gap,
and the analysis below will show that it converges linearly to zero.
In practice, it also provides a convenient stopping criterion, which 
yields Algorithm~\ref{alg:miso2}.

\begin{algorithm}
   \caption{MISO-Prox with stopping criterion.}\label{alg:miso2}
    \begin{algorithmic}[1]
       \INPUT $(z_i^0,c_i^0)_{i=1,\ldots,n}$ such that~(\ref{A1}) holds;  $\varepsilon$ (target accuracy);
       \STATE initialize $\barz_0= \frac{1}{n}\sum_{i=1}^n z_i^0$ and~$c_i^{\prime 0} = c_i^0 + \frac{\mu}{2}\|\barz_0\|^2$ for all $i$ in~$\{1,\ldots,n\}$ and $x_0 = \prox_{\psi/\mu}[\barz_0]$;
       \STATE Define $\delta = \min\left(1,\frac{\mu n}{2(L-\mu)}\right)$ and~$k=0$;
       \WHILE{ $\frac{1}{n}\sum_{i=1}^n f_i(x_k) - c_i^{\prime k} + \mu\langle \barz_k, x_k\rangle - \frac{\mu}{2}\|x_k\|^2 > \varepsilon$ }
       \FOR{ $l = 1,\ldots,n$ } 
       \STATE $k \leftarrow k+1$;
       \STATE randomly pick up an index~$i_k$ in~$\{1,\ldots,n\}$;
       \STATE perform the update~(\ref{eq:misoupdate});
       \STATE update 
       \begin{equation}\label{eq:cprime}
          c_i^{\prime k} = \left\{  \begin{array}{ll}
                (1-\delta) c_i^{\prime k-1} + \delta \left( f_i(x_{k-1}) - \langle \nabla f_i(x_{k-1}), x_{k-1} \rangle +\frac{\mu}{2}\|x_{k-1}\|^2  \right) & \text{if}~~~ i=i_k \\
                c_i^{\prime k-1} & \text{otherwise}
             \end{array} \right. .
       \end{equation}
       \ENDFOR
       \ENDWHILE
       \OUTPUT $x_N$ (final estimate such that~$F(x_N)-F^\star \leq \varepsilon$).
\end{algorithmic}
\end{algorithm}
To explain the stopping criterion in Algorithm~\ref{alg:miso2}, we remark that the functions~$d_i^k$ are quadratic and can be written
\begin{equation} \label{eq:dk2}
   d_i^k(x) =   c_i^k + \frac{\mu}{2}\|x-z_{i}^k\|^2 = c_i^{\prime k} - \mu \langle x,z_i^k \rangle + \frac{\mu}{2}\|x\|^2,
\end{equation}
where the~$c_i^k$'s are some constants and $c_i^{\prime k} = c_i^k + \frac{\mu}{2}\|z_i^k\|^2$.
Equation~(\ref{eq:cprime}) shows how to update recursively these constants~$c_i^{\prime k}$, and finally
\begin{displaymath}
   D_k(x_k) = \left(\frac{1}{n}\sum_{i=1}^n c_i^{\prime k}\right) - \mu \langle x_k, \barz_k \rangle + \frac{\mu}{2}\|x_k\|^2 + \psi(x_k),
\end{displaymath}
and
\begin{displaymath}
   F(x_k) - D_k(x_k)  = \left(\frac{1}{n}\sum_{i=1}^n f_i(x_k)- c_i^{\prime k}\right) + \mu \langle x_k, \barz_k \rangle - \frac{\mu}{2}\|x_k\|^2,
\end{displaymath}
which justifies the stopping criterion. Since computing $F(x_k)$ requires
scanning all the data points, the criterion is only computed every~$n$
iterations. 

\paragraph{Convergence Analysis.}
The convergence of MISO-Prox is guaranteed by Theorem~\ref{thm1.1} from
the main part of paper.
Before we prove this theorem, we note that this rate is slightly better than
the one proven in MISO~\cite{miso}, which converges as $(1-\frac{1}{3n})^k$.
We start by recalling a classical lemma that provides useful inequalities. Its proof 
may be found in~\cite{nesterov}.

\begin{lemma}[\bfseries Classical Quadratic Upper and Lower Bounds]\label{rem1.3}~\newline
   For any function~$g: \Real^p \to \Real$ which is $\mu$-strongly convex and differentiable with $L$-Lipschitz derivatives, we have for all~$x,y$ in~$\Real^p$,
   $$  \frac{\mu}{2} \Vert x- y \Vert^2 \leq g(x) - g(y) + \langle \nabla g (y), x - y \rangle \leq \frac{L}{2} \Vert x- y \Vert^2. $$
\end{lemma}

To start the proof, we need a sequence of upper and lower bounds involving the functions~$D_{k}$ and $D_{k-1}$. The first one is given in the next lemma
\begin{lemma}[\bfseries Lower Bound on~$D_k$]~\newline
  For all~$k \geq 1$ and~$x$ in~$\Real^p$,
   \begin{equation}\label{eq6}
      \Surr_k(x) \geqslant \Surr_{k-1}(x) - \frac{\delta (L-\mu) }{2n} \Vert x - x_{k-1} \Vert^2, \quad \forall x \in \mathbb{R}^p.
   \end{equation}
\end{lemma}

\begin{proof}
For any $i\in \{1,\ldots,n\}$, $f_i$ satisfies the assumptions of Lemma~\ref{rem1.3}, and we have for all~$k \geq 0$,~$x$ in~$\Real^p$, and for~$i=i_k$,
   \begin{align*}
      \surr_{i}^k (x) &= (1-\delta) \surr_i^{k-1}(x) + \delta [ f_i(x_{k-1}) + \langle \nabla f_i (x_{k-1}), x - x_{k-1} \rangle + \frac{\mu}{2} \Vert x- x_{k-1} \Vert^2]  \\
                                             &\geqslant (1-\delta) \surr_i^{k-1}(x) + \delta f_i(x) -  \frac{\delta  (L-\mu)}{2} \Vert x - x_{k-1} \Vert^2 \\
                                                                  & \geqslant \surr_i^{k-1}(x) -  \frac{\delta  (L-\mu)}{2} \Vert x - x_{k-1} \Vert^2,
   \end{align*}
   where the definition of $d_i^k$ is given in~(\ref{eq:dk}). The first inequality uses Lemma~\ref{rem1.3}, and the last one uses the inequality~$f_i \geq d_i^{k-1}$. 
   From this inequality, we can obtain~(\ref{eq6}) by simply using
   $\Surr_k(x) = \sum_{i=1}^n \surr_i^k(x) + \reg(x) = \Surr_{k-1}(x) + \frac{1}{n} \left (\surr^k_{\, i_k}(x) - \surr^{k-1}_{\, i_k}(x) \right )$.
\end{proof}
Next, we prove the following lemma to compare $\Surr_k$ and $\Surr_{k-1}$.
\begin{lemma}[\bfseries Relation between~$D_k$ and~$D_{k-1}$]\label{1.3}~\newline
   For all $k \geq 0$, for all $x$ and~$y$ in $\mathbb{R}^p$, 
   \begin{equation}\label{D=H}
      \Surr_k(x) - \Surr_k(y) = \Surr_{k-1}(x) - \Surr_{k-1}(y) - \mu \langle \barz_k - \barz_{k-1}, x-y \rangle.
   \end{equation}
\end{lemma}
\begin{proof}
   Remember that the functions~$d_i^k$ are quadratic and have the form~(\ref{eq:dk2}), that~$D_k$ is defined in~(\ref{eq:Dk}), and that~$\barz_k$ minimizes $\frac{1}{n}\sum_{i=1}^n d_i^k$.
   Then, there exists a constant $A_k$ such that
   $$ \Surr_k(x) = A_k + \frac{\mu}{2} \Vert x-\barz_k \Vert^2 + \reg(x).$$
   This gives 
   \begin{equation}\label{1.3.1}
      \Surr_k(x) - \Surr_k(y) = \frac{\mu}{2} \Vert x-\barz_k \Vert^2 - \frac{\mu}{2} \Vert y-\barz_k \Vert^2 + \reg(x) - \reg(y). 
   \end{equation}
   Similarly,
   \begin{equation}\label{1.3.2}
      \Surr_{k-1}(x) - \Surr_{k-1}(y) = \frac{\mu}{2} \Vert x-\barz_{k-1} \Vert^2 - \frac{\mu}{2} \Vert y-\barz_{k-1} \Vert^2 + \reg(x) - \reg(y) .
   \end{equation}
   Subtracting (\ref{1.3.1}) and (\ref{1.3.2}) gives (\ref{D=H}). 
\end{proof}
Then, we are able to control the value of $\Surr_{k}(x_{k-1})$ in the
next lemma.
\begin{lemma}[\bfseries Controlling the value~$D_k(x_{k-1})$]\label{1.4}~\newline
   For any $k \geqslant 1$, 
   \begin{equation} \label{eq11}
      \Surr_{k}(x_{k-1}) - \Surr_{k} (x_{k})  \leqslant  \frac{\mu}{2} \Vert \barz_k - \barz_{k-1} \Vert^2.
   \end{equation}
\end{lemma}
\begin{proof}
   Using Lemma~\ref{1.3} with $x = x_{k-1}$ and $y = x_{k}$ yields
   $$ \Surr_k(x_{k-1}) - \Surr_k (x_{k}) = \Surr_{k-1}(x_{k-1}) - \Surr_{k-1}(x_{k}) - \mu \langle \barz_k - \barz_{k-1}, x_{k-1} - x_{k} \rangle.  $$
   Moreover $x_{k-1}$ is the minimum of $\Surr_{k-1}$ which is $\mu$-strongly convex. Thus,
   $$ \Surr_{k-1}(x_{k-1}) + \frac{\mu}{2} \Vert x_k - x_{k-1} \Vert^2 \leqslant \Surr_{k-1}(x_k).$$
   Adding the two previous inequalities gives the first inequality below
   $$  \Surr_k(x_{k-1}) - \Surr_k (x_{k}) \leqslant  -\frac{\mu}{2} \Vert x_k - x_{k-1} \Vert^2- \mu \langle \barz_k - \barz_{k-1}, x_{k-1} - x_{k} \rangle \leqslant  \frac{\mu}{2} \Vert \barz_k - \barz_{k-1} \Vert^2,  $$
   and the last one comes from the basic inequality $\frac{1}{2}\|a\|^2 + \langle a,b\rangle + \frac{1}{2}\|b\|^2 \geq 0$.
\end{proof}

We have now all the inequalities in hand to prove  Theorem~\ref{thm1.1}.

\begin{proof}[Proof of Theorem~\ref{thm1.1}]~\newline
   We start by giving a lower bound of $\Surr_{k}(x_{k-1}) - \Surr_{k-1}(x_{k-1}) $. \\ 
   Take $x = x_{k-1}$ in (\ref{D=H}). Then, for all~$y$ in~$\Real^p$,  
   \begin{displaymath}
      \begin{split}
       \Surr_{k}(x_{k-1}) - \Surr_{k-1}(x_{k-1})  =& \, \Surr_{k}(y) - \Surr_{k-1}(y) +  \mu \langle \barz_{k}-\barz_{k-1}, y- x_{k-1} \rangle \\
      by \,\, (\ref{eq6})  \geqslant& - \frac{ \delta (L-\mu)}{2n} \Vert y - x_{k-1}\Vert^2 + \mu \langle \barz_{k}-\barz_{k-1}, y- x_{k-1} \rangle  
      \end{split}
   \end{displaymath}
   Choose $y$ that maximizes the above quadratic function, i.e.
   $$ y = x_{k-1} + \frac{n \mu}{\delta (L- \mu)} (\barz_k - \barz_{k-1}),  $$
   and then
   \begin{equation}
      \begin{split}
         \Surr_{k}(x_{k-1}) - \Surr_{k-1}(x_{k-1})    & \geqslant  \frac{n \mu^2}{2 \delta (L-\mu)} \Vert \barz_k - \barz_{k-1} \Vert^2  \\
         by \,\,(\ref{eq11}) & \geqslant \frac{n \mu}{ \delta (L-\mu)}  \left [ \Surr_{k}(x_{k-1}) - \Surr_k(x_k)  \right ] . \label{eq:misoconv2}
      \end{split}
   \end{equation}
   Then, we start introducing expected values. \\
   By construction 
   $$\Surr_{k} (x_{k-1}) = \Surr_{k-1} (x_{k-1}) + \frac{\delta}{n} (f_{i_k} (x_{k-1}) - \surr_{i_k}^{k-1} (x_{k-1})). $$
   After taking expectation, we obtain the relation 
   \begin{equation}\label{eq:misoconv}
      \mathbb{E}[ \Surr_k(x_{k-1})] = \left(1- \frac{\delta}{n}\right) \mathbb{E}[ \Surr_{k-1}(x_{k-1})] + \frac{\delta}{n} \mathbb{E} [F(x_{k-1})].
   \end{equation}
   We now introduce an important quantity
   $$\tau = \left(1- \frac{\delta (L-\mu)}{n \mu}\right) \frac{\delta}{n},$$
   and combine~(\ref{eq:misoconv2}) with~(\ref{eq:misoconv}) to obtain
    \begin{equation*} \label{eq13}
       \tau \mathbb{E} [F(x_{k-1})]  - \mathbb{E}[ \Surr_k(x_k)] \leqslant -(1-\tau) \mathbb{E}[\Surr_{k-1} (x_{k-1})].
   \end{equation*}
   We reformulate this relation as 
   \begin{equation} \label{eq14}
      \tau \left ( \mathbb{E} [F(x_{k-1})] -F^* \right ) + \left ( F^* - \mathbb{E}[ \Surr_k(x_k)] \right ) \leqslant (1-\tau) \left ( F^* -  \mathbb{E}[\Surr_{k-1} (x_{k-1})] \right ).
   \end{equation}
   On the one hand, since $ F(x_{k-1}) \geqslant F^*$, we have 
   $$   F^* - \mathbb{E}[ \Surr_k(x_k)]  \leqslant (1-\tau) \left ( F^* -  \mathbb{E}[\Surr_{k-1} (x_{k-1})] \right ).$$
   This is true for any $k \geqslant 1$, as a result 
   \begin{equation}\label{eq15}
      F^* - \mathbb{E}[ \Surr_k(x_k)]  \leqslant (1-\tau)^k \left ( F^* -  \Surr_{0} (x_{0}) \right ).
   \end{equation}
   On the other hand, since $F^* \geqslant \Surr_k(x_k)$, then 
   $$\tau \left ( \mathbb{E} [F(x_{k-1})] -F^* \right )\leqslant (1-\tau) \left ( F^* -  \mathbb{E}[\Surr_{k-1} (x_{k-1})] \right ) \leqslant (1-\tau)^k \left ( F^* -  \Surr_{0} (x_{0}) \right )
   ,$$ 
   which gives us the relation~(\ref{eq:thm_miso}) of the theorem.  We conclude
   giving the choice of $\delta$. We choose it to maximize the rate of
   convergence, which turns to maximize $\tau$.
   This is 
   a quadratic function, which is maximized at $\delta=\frac{n \mu}{2 (L-\mu)}$. However,  by definition $\delta \leqslant 1$. Therefore, the optimal choice of $\delta$ is given by 
   $$ \delta = \min \Big \{1, \frac{n \mu}{2 (L-\mu)}  \Big \}.$$
   Note now that
   \begin{enumerate}
      \item When $ \frac{n \mu}{2 (L-\mu)}  \leqslant 1$, we have $\delta = \frac{n \mu}{2 (L-\mu)}$ and $\tau = \frac{\mu}{4(L-\mu)}$. 
      \item When $ 1 \leqslant \frac{n \mu}{2 (L-\mu)}  $,  we have $\delta =1$ and $\tau =  \frac{1}{n}  - \frac{L-\mu}{n^2 \mu}  \geqslant \frac{1}{2n} $.
   \end{enumerate}
Therefore,~$\tau \geq \min\left(\frac{1}{2n},\frac{\mu}{4(L-\mu)}\right)$, which concludes the first part of the theorem.

To prove the second part, we use (\ref{eq15}) and (\ref{eq:thm_miso}), which gives
   \begin{align*}
      \mathbb{E}[F(x_{k}) - \Surr_k(x_k)]= \,\,& \mathbb{E}[F(x_{k})]-F^*+ F^* - \mathbb{E}[\Surr_k(x_k)]  \\
      \leqslant \,\,& \frac{1}{\tau} (1-\tau)^{k+1}( F^* -  \Surr_{0} (x_{0}) + (1-\tau)^k ( F^* -  \Surr_{0} (x_{0})) \\
      = \,\, & \frac{1}{\tau} (1-\tau)^k ( F^* -  \Surr_{0} (x_{0})).
   \end{align*}
\end{proof}

\subsection{Accelerating MISO-Prox}
The convergence rate of MISO (or also SDCA) requires a special handling since it does
not satisfy exactly the condition~(\ref{eq:rategk}) from
Proposition~\ref{prop:complexity}. The rate of convergence is linear, but with
a constant proportional to~$F^\star-D_0(x_0)$ instead of~$F(x_0)-F^\star$ for
many classical gradient-based approaches.

To achieve acceleration, we show in this section how to obtain similar guarantees as Proposition~\ref{prop:complexity} and~\ref{prop:complexity2}---that is, how to solve efficiently the subproblems~(\ref{eq:approx}). This essentially requires the right initialization each time MISO-Prox is called. By initialization, we mean initializing the variables~$z_i^0$.

Assume that MISO-Prox is used to obtain $x_{k-1}$
from~Algorithm~\ref{alg:catalyst} with $G_{k-1}(x_{k-1})-G_k^\star \leq
\varepsilon_{k-1}$, and that one wishes to use MISO-Prox again on~$G_k$ to
compute~$x_k$. Then, let us call $D'$ the lower-bound of~$G_{k-1}$ produced by MISO-Prox when computing~$x_{k-1}$ such that
$$ x_{k-1} = \argmin_{x \in \mathbb{R}^p} \left \{ D'(x)   = \frac{1}{n}\sum_{i=1}^n d'_i(x) + \psi(x) \right \}, $$
with 
$$  d_i'(x) = \frac{\mu+\kappa}{2} \Vert x - z_i' \Vert^2 + c_i'.$$
Note that we do not index these quantities with $k-1$ or~$k$ for the sake of simplicity.
The convergence of MISO-Prox may ensure that not only do we have
$G_{k-1}(x_{k-1})-G_k^\star \leq \varepsilon_{k-1}$, but in fact we have the stronger condition
$G_{k-1}(x_{k-1})-D'(x_{k-1}) \leq \varepsilon_{k-1}$.
Remember now that
$$ G_k(x) = G_{k-1} (x) + \frac{\kappa}{2} \Vert x - y_{k-1} \Vert^2 - \frac{\kappa}{2} \Vert x - y_{k-2} \Vert^2,$$
and that~$D'$ is a lower-bound of~$G_{k-1}$. Then, we may set for all~$i$ in~$\{1,\ldots,n\}$
$$ d_i^0(x) = d'_i(x) + \frac{\kappa}{2} \Vert x-y_{k-1} \Vert^2 - \frac{\kappa}{2} \Vert x -y_{k-2} \Vert^2,$$
which is equivalent to initializing the new instance of MISO-Prox with
$$ z_i^0 = z'_i + \frac{\kappa}{\kappa+\mu}(y_{k-1}-y_{k-2}),$$
and by choosing appropriate quantities~$c_i^0$.
Then, the following function is a lower bound of~$G_k$
$$
D_0(x) = \frac{1}{n} \sum_{i=1}^n d_i^0(x) + \psi(x).
$$
and the new instance of MISO-Prox to minimize $G_k$ and compute~$x_k$ will 
produce iterates, whose first point, which we
call~$x^0$, minimizes~$D_0$. This leads to the relation
$$
x^0 = \prox_{\psi/{(\kappa+\mu)}}\left[\barz^0\right] = \prox_{\psi/{(\kappa+\mu)}}\left[\barz' + \frac{\kappa}{\kappa+\mu}(y_{k-1}-y_{k-2})\right],
$$
where we use the notation $\barz^0 = \frac{1}{n}\sum_{i=1}^n{z^0_i}$ and $\barz' = \frac{1}{n}\sum_{i=1}^n{z'_i}$ as in Algorithm~\ref{alg:miso}.

Then, it remains to show that the quantity~$G_k^\star - D_0(x^0)$ is upper
bounded in a similar fashion as $G_k(x_{k-1})-G_k^\star$ in Propositions~\ref{prop:complexity} and~\ref{prop:complexity2} 
to obtain a similar result for MISO-Prox and control the number of inner-iterations.
This is indeed the case, as stated in the next lemma.
\begin{lemma}[\bfseries Controlling $G_k(x_{k-1})-G_k^\star$ for MISO-Prox]~\label{lemma:accproxmiso}\newline
   When initializing MISO-Prox as described above, we have
   \begin{displaymath}
      G_k^* -  D_{0}(x^{0}) \leqslant \varepsilon_{k-1} +  \frac{\kappa^2}{2 (\kappa+\mu)} \Vert y_{k-1} - y_{k-2} \Vert^2.
   \end{displaymath}
\end{lemma}
\begin{proof}
By strong convexity, we have
$$ D_0(x^0) + \frac{\kappa}{2} \Vert x^0-y_{k-2} \Vert^2 - \frac{\kappa}{2} \Vert x^0 -y_{k-1} \Vert^2 =D'_0(x^0) \geqslant  D'_0(x_{k-1}) + \frac{\kappa+\mu}{2} \Vert x^0 - x_{k-1} \Vert^2. $$
Consequently, 
\begin{align*}
   D_0(x^0) \geqslant \,\,& D'(x_{k-1}) - \frac{\kappa}{2} \Vert x^0-y_{k-2} \Vert^2 + \frac{\kappa}{2} \Vert x^0 -y_{k-1} \Vert^2 + \frac{\kappa+\mu}{2} \Vert x^0 - x_{k-1} \Vert^2\\
   = \,\, & D_0(x_{k-1})+  \frac{\kappa}{2} \Vert x_{k-1}-y_{k-2} \Vert^2 - \frac{\kappa}{2} \Vert x_{k-1} -y_{k-1} \Vert^2  - \frac{\kappa}{2} \Vert x^0-y_{k-2} \Vert^2 + \frac{\kappa}{2} \Vert x^0 -y_{k-1} \Vert^2 \\
                        & + \frac{\kappa+\mu}{2} \Vert x^0 - x_{k-1} \Vert^2\\
   =\,\,& D_0(x_{k-1}) - \kappa \langle x^0 - x_{k-1}, y_{k-1} - y_{k-2} \rangle + \frac{\kappa+\mu}{2} \Vert x^0 - x_{k-1} \Vert^2 \\
   \geqslant \,\, &  D_0(x_{k-1}) - \frac{\kappa^2}{2 (\kappa+\mu)} \Vert y_{k-1} - y_{k-2} \Vert^2,
\end{align*}
where the last inequality is using a simple relation $\frac{1}{2}\|a\|^2+2 \langle a,b\rangle +\frac{1}{2}\|b\|^2 \geq 0$.
As a result,
\begin{eqnarray*}
   G_k^* -  D_{0}(x^{0}) &\leqslant & G_k^* -  D_{0}(x_{k-1}) +  \frac{\kappa^2}{2 (\kappa+\mu)} \Vert y_{k-1} - y_{k-2} \Vert^2 \\ 
                                                     &\leqslant& G_k(x_{k-1}) -  D_{0}(x_{k-1}) +  \frac{\kappa^2}{2 (\kappa+\mu)} \Vert y_{k-1} - y_{k-2} \Vert^2 \\
                                                                                 &=& G_{k-1}(x_{k-1}) -  D'(x_{k-1})+  \frac{\kappa^2}{2 (\kappa+\mu)} \Vert y_{k-1} - y_{k-2} \Vert^2  \\
                                                                                                             &\leq & \varepsilon_{k-1} +  \frac{\kappa^2}{2 (\kappa+\mu)} \Vert y_{k-1} - y_{k-2} \Vert^2 
\end{eqnarray*}
\end{proof}
{ 
We remark that this bound is half of the bound shown in (\ref{bound}). Hence, a similar argument gives the bound on the number of inner iterations. 
We may finally compute the iteration-complexity of accelerated MISO-Prox.}
\begin{prop}[Iteration-Complexity of Accelerated MISO-Prox]~\newline
   When~$F$ is~$\mu$-strongly convex, the accelerated MISO-Prox algorithm achieves the accuracy~$\varepsilon$ with an expected number of iteration upper bounded by
$$ O \left ( \min \left \{  \frac{L}{\mu}, \sqrt{ \frac{ nL}{ \mu} } \right \} \log \left ( \frac{1}{\epsilon}  \right ) \log \left ( \frac{L}{\mu} \right ) \right ).  $$
\end{prop}
\begin{proof}
{ 
   When $n >2(L-\mu)/\mu$, there is no acceleration. The optimal value for~$\kappa$ is zero, and we may use Theorem~\ref{thm1.1} and Lemma~\ref{lemma:expectation} to obtain the
   complexity
$$ O \left (\frac{L}{\mu} \log \left (  \frac{L}{\mu}\frac{F(x_0)-D_0(x_0)}{\epsilon}  \right ) \right). $$
When $n < 2(L-\mu)/\mu$, there is an acceleration, with $\kappa = 2(L-\mu)/\mu \; - \mu$. Let us compute the global complexity using the ``template'' presented in Appendix~\ref{sec:accel-factor-calculs}. }The number of outer iteration is given by 
$$ k_{\text{out}} = O \left ( \sqrt{ \frac{ L}{n \mu} } \log \left ( \frac{F(x_0)- F^*}{\epsilon} \right )  \right ).  $$
At each inner iteration, we initialize with the value $x^0$ described above, and we use Lemma~\ref{lemma:accproxmiso}:
$$ G_k^* - D_0(x^0) \leq \epsilon_{k-1}+ \frac{\kappa}{2} \Vert y_{k-1} - y_{k-2} \Vert^2.  $$
Then,
$$ \frac{G_k^* - D_0(x^0)}{\epsilon_k} \leq \frac{R}{2}, $$
where 
$$ R = \frac{2}{1-\rho}+ \frac{2592 \kappa}{\mu (1- \rho)^2 (\sqrt{q}-\rho)^2} = O\left (  \left (\frac{L}{n\mu} \right )^2 \right). $$
With Miso-Prox, with have $\tau_{G_k} = \frac{1}{2n}$, thus the expected number of inner iteration is given by Lemma~\ref{lemma:expectation}:
{ $$ k_{\text{in}} = O(n \log(n^2 R)) = O \left ( n \log \left ( \frac{L}{\mu} \right ) \right ).  $$}
As a result,
$$  \text{Comp} = O \left ( \sqrt{ \frac{ nL}{ \mu} } \log \left ( \frac{F(x_0)- F^*}{\epsilon}  \right ) \log \left ( \frac{L}{\mu} \right ) \right ). $$
To conclude, the complexity of the accelerated algorithm is given by
$$ O \left ( \min \left \{  \frac{L}{\mu}, \sqrt{ \frac{ nL}{ \mu} } \right \} \log \left ( \frac{1}{\epsilon}  \right ) \log \left ( \frac{L}{\mu} \right ) \right ).  $$
\end{proof}

\section{Implementation Details of Experiments}\label{appendix:exp}
{In the experimental section, we compare the performance with and without acceleration for three algorithms SAG, SAGA and MISO-Prox on $l_2$-logistic regression problem. In this part, we clarify some details about the implementation of the experiments. 

Firstly, we normalize the observed data before running the regression. Then we apply Catalyst using parameters according to the theoretical settings. Standard analysis of the logistic function shows that the Lipschitz gradient parameter $L$ is $1/4$ and strongly convex parameter $\mu=0$ when there is no regularization. Adding properly a $l_2$ term generates the strongly-convex regimes. Several parameters need to be fixed at the beginning stage. The parameter $\kappa$ is set to {its optimal value suggested by theory}, which only depends on $n$, $\mu$ and $L$. More precisely, $\kappa$ writes as $\kappa  = {a(L-\mu)}/{(n+b)} -\mu$, with $(a,b) = (2, -2)$ for SAG, $(a,b) = (1/2, 1/2)$ for SAGA and $(a,b) =(1,1)$ for MISO-Prox. The parameter $\alpha_0$ is initialized as the positive solution of $x^2+ (1-q) x -1 =0$ where $q= \sqrt{\mu/ (\mu+\kappa)}$. Furthermore, since the objective function is always positive, $F(x_0) -F^*$ can be upper bounded by $F(x_0)$ which allow us to set the $\varepsilon_k = (2/9) F(x_0)(1-\rho)^k$ in the strongly convex case and $\varepsilon_k = 2F(x_0)/9(k+2)^{4+\eta}$ in the non-strongly convex case. 
{
Finally, we set the free parameter in the expression of $\varepsilon_k$ as follows. We simply set $\rho = 0.9\sqrt{q}$ in the strongly convex case and $\eta= 0.1$ in the non strongly convex case.}

To solve the subproblem at each iteration, the step-sizes parameter for SAG, SAGA and MISO are set to {the values suggested by theory}, which only depend on $\mu$, $L$ and $\kappa$. {All of the methods we compare store~$n$ gradients evaluated at previous iterates of the algorithm. For MISO, the convergence analysis in Appendix~\ref{appendix:miso} leads to the initialization $x_{k-1}+\frac{\kappa}{\mu+\kappa}(y_{k-1}-y_{k-2})$ that moves $x_{k-1}$ closer to $y_{k-1}$ and further away from~$y_{k-2}$. We found that using this initial point for SAGA was giving slightly better results than~$x_{k-1}$.}

}

\end{document}